\documentclass[12pt,twoside,reqno]{amsart}
\usepackage{amsmath}
\usepackage{amsfonts}
\usepackage{amssymb}
\usepackage{color}
\usepackage{mathrsfs}
\usepackage{multicol}
\usepackage{cite}
\usepackage{geometry}
\usepackage{marginnote}
\usepackage{todonotes}
\usepackage{tikz}
\usepackage{pgfplots}
\pgfplotsset{compat=1.10}
\usepgfplotslibrary{fillbetween}
\usetikzlibrary{patterns}
\newcommand{\real}{\mathbb R}

\definecolor{darkred}{rgb}{0.8,0,0}

\newtheorem{theorem}{Theorem}[section]
\newtheorem{corollary}[theorem]{Corollary}
\newtheorem{lemma}[theorem]{Lemma}
\newtheorem{proposition}[theorem]{Proposition}
\newtheorem{definition}[theorem]{Definition}
\newtheorem{remark}[theorem]{Remark}
\numberwithin{equation}{section}
\begin{document}
\title{A parabolic Hardy-H\'enon equation\\ with quasilinear degenerate diffusion}
\thanks{}
\author{Razvan Gabriel Iagar}
\address{Departamento de Matem\'{a}tica Aplicada, Ciencia e Ingenieria de los Materiales y Tecnologia Electr\'onica, Universidad Rey Juan Carlos, M\'{o}stoles, 28933, Madrid, Spain}
\email{razvan.iagar@urjc.es}
\author{Philippe Lauren\c{c}ot}
\address{Laboratoire de Math\'ematiques (LAMA) UMR 5217, Universit\'e Savoie-Mont Blanc, CNRS, F-73000, Chamb\'ery, France}
\email{philippe.laurencot@univ-smb.fr}


\date{\today}

\begin{abstract}
Local and global well-posedness, along with finite time blow-up, are investigated for the following Hardy-H\'enon equation involving a quasilinear degenerate diffusion and a space-dependent superlinear source featuring a singular potential
\begin{equation*}
	\partial_t u=\Delta u^m+|x|^{\sigma}u^p, \qquad t>0, \ x\in\real^N,
\end{equation*}
when $m>1$, $p>1$ and $\sigma\in \big(\max\{-2,-N\},0 \big)$. While the superlinear source induces finite time blow-up when $\sigma=0$, whatever the value of $p>1$, at least for sufficiently large initial conditions, a striking effect of the singular potential $|x|^\sigma$ is the prevention of finite time blow-up for suitably small values of $p$, namely, $1<p\le p_G := [2-\sigma(m-1)]/2$. Such a result, as well as the local existence of solutions for $p>p_G$, is obtained by employing the Caffarelli-Kohn-Nirenberg inequalities. Another interesting feature is that uniqueness and comparison principle hold true for generic non-negative initial conditions when $p>p_G$, but their validity is restricted to initial conditions which are positive in a neighborhood of $x=0$ when $p\in (1,p_G)$, a range in which non-uniqueness holds true without this positivity condition. Finite time blow-up of any non-trivial, non-negative solution is established when $p_G<p\leq p_F:=m+(\sigma+2)/N$, while global existence for small initial data in some critical Lebesgue spaces and blow-up in finite time for initial data with a negative energy are proved for $p>p_F$. Optimal temporal growth rates are also derived for global solutions when $p\in (1,p_G]$. All the results are sharp with respect to the exponents $(m,p,\sigma)$ and conditions on $u_0$.
\end{abstract}

\maketitle

%
%
\pagestyle{myheadings}
\markboth{\sc{R.G. Iagar \& Ph. Lauren\c cot}}{\sc{Quasilinear Hardy-H\'enon Equation}}

\section{Introduction}\label{sec.int}

The goal of this work is to study the well-posedness and the maximal existence time of non-negative solutions to the following Cauchy problem
\begin{subequations}\label{cp}
\begin{align}
	\partial_t u & = \Delta u^m + |x|^\sigma u^p, \qquad t>0, \ x\in\real^N, \label{eq1}\\
	u(0) & = u_0, \qquad \ x\in\real^N, \label{ic}
\end{align}
\end{subequations}
when the initial condition $u_0$ lies in the positive cone of the Lebesgue space $L^r(\mathbb{R}^N)\cap L^\infty(\mathbb{R}^N)$ for appropriately chosen values of $r\ge 1$ and the exponents $(m,p,\sigma)$ satisfy
\begin{equation}
	m>1, \quad p>1, \quad \sigma\in \big(\max\{-2,-N\},0 \big). \label{exp}
\end{equation}
We observe that Eq.~\eqref{eq1} involves a competition between a quasilinear diffusion term (of porous medium type) and a superlinear source term featuring a potential which becomes singular at $x=0$, owing to the negativity of $\sigma$. Already the competition between diffusion and reaction is a very rich source of interesting mathematical phenomena according to the monographs \cite{QuSo2019, SGKM1995}, the former dealing with the semilinear diffusion $m=1$ and the latter with the quasilinear diffusion $m>1$, but, as we prove in the present paper, the presence of a singular potential weighting the reaction term leads to an interesting novelty. Indeed, we establish below that, if $p\in \big(1,(2-\sigma(m-1))/2\big]$, then all non-negative solutions to the Cauchy problem~\eqref{cp} are global, a situation which contrasts markedly with the case $\sigma=0$, for which there are always non-negative solutions blowing up in finite time whatever the value of $(m,p)\in (1,\infty)^2$.

\medskip

Equations featuring a diffusion term and a source term involving a singular weight such as~\eqref{eq1} are usually referred to in the literature as parabolic \emph{Hardy-H\'enon equations}. This name stems on the one hand from \cite[Eq.~(A.6)]{He73} where the elliptic counterpart of Eq.~\eqref{eq1} is proposed as a model for studying rotating stellar systems, and on the other hand from \cite{BG84}, where the heat equation with a potential $|x|^{-2}$ is considered and the optimal constant of the Hardy inequality proves to be the threshold between local existence and non-existence (in the form of an instantaneous blow-up) of solutions. Thus, the names of Hardy equation, respectively H\'enon equation became usual to refer to equations featuring potentials $|x|^{\sigma}$ with $\sigma<0$, respectively with $\sigma>0$, although in \cite[Eq.~(A.6)]{He73} singular potentials with $\sigma>-2$ are considered as well.

In recent years, a number of works contributed to the development of the mathematical theory of the parabolic Hardy-H\'enon equation with semilinear diffusion
\begin{equation}\label{eq0}
	\partial_t u=\Delta u+|x|^{\sigma}u^p, \qquad t>0, \ x\in\real^N,
\end{equation}
with $p>1$ and $\sigma$ as in~\eqref{exp} (in some cases $\sigma>0$ being also considered). Thus, local well-posedness for~\eqref{eq0} in $L^q(\mathbb{R}^N)$ and in $C_0(\mathbb{R}^N)$ is established in \cite{BSTW17}, together with the large time behavior as $t\to\infty$ of its solutions when they are global in time, depending on the properties of $u_0(x)$ as $|x|\to\infty$. This well-posedness has been improved both in the sense of extending the class of initial conditions to data $u_0$ that are Radon measures or functions featuring singular points \cite{HS24, HT21} or by optimizing the Lebesgue or Lorentz spaces for well-posedness, see for example the series of works \cite{CIT21, CIT22, CITT24} and the references therein, where a study of the long-term dynamics of~\eqref{eq0} is also performed. Some of the above mentioned works extend the study to equations involving fractional diffusion, but a common tool in all of them is the representation formula using the heat kernel in order to derive estimates on the solution.

Letting $m>1$ in the diffusion operator as in Eq.~\eqref{eq1} brings a significant number of novelties, both at qualitative and technical levels. Regarding the qualitative aspects, one of the authors and his collaborators performed a rather complete classification of the possible behaviors of radially symmetric self-similar solutions to Eq.~\eqref{eq1} with $m$, $p$, $\sigma$ as in~\eqref{exp}, with the extra condition $1<p<m$, in recent works, where it has been noticed, among other properties, that the form and properties of such self-similar solutions to Eq.~\eqref{eq1} strongly depend on the sign of $p-p_G$ with
\begin{equation}
	p_G := 1 - \frac{\sigma(m-1)}{2} \in (1,m). \label{pLexp}
\end{equation}
Indeed, in the range $p\in (p_G,m)$, considered in \cite{ILS24a}, typical self-similar solutions are in backward form
\begin{equation}\label{backward}
u(t,x)=(T-t)^{-\alpha}f(|x|(T-t)^{\beta}), \quad \alpha=\frac{\sigma+2}{2(p-p_G)}, \quad \beta=\frac{m-p}{2(p-p_G)}, \quad T\in(0,\infty),
\end{equation}
presenting thus a finite time blow-up at $t=T$. On the contrary, in the opposite case $p\in (1,p_G)$, a unique, compactly supported, radially symmetric solution in forward form (which is global in time)
\begin{equation}\label{forward}
 \mathcal{U}_*(t,x)=t^{\alpha_*}f_*(|x|t^{-\beta_*}), \quad \alpha_*=-\frac{\sigma+2}{2(p-p_G)}, \quad \beta_*=-\frac{m-p}{2(p-p_G)},
\end{equation}
is constructed in \cite{IMS23}, establishing at the same time an example of \emph{strong non-uniqueness of solutions}, since the solution in~\eqref{forward} has $u_0\equiv0$ as initial trace. Finally, radially symmetric self-similar solutions to Eq. \eqref{eq1} in the critical case $p=p_G$ are classified in \cite{ILS24b}, where it is shown that they are necessarily in exponential form
\begin{equation}\label{exponential}
\mathcal{U}^*(t,x)=e^{\alpha^* t} f^*\big(|x|e^{-\beta^* t}\big),
\end{equation}
for some unique exponents $(\alpha^*,\beta^*)\in(0,\infty)^2$ which are not explicit, but satisfy $2\beta^*=(m-1)\alpha^*$. As we shall see below, all these self-similar solutions are weak solutions to~\eqref{eq1} in the sense given in Definition~\ref{def.ws} and Corollary~\ref{cor.exist} below. Such a stark difference is not seen at the level of the semilinear equation~\eqref{eq0}, since for $m=1$, $p_G=1$ and $p-p_G=p-1$ is always positive. The above discussion leads us to expect some significant qualitative differences in the mathematical analysis of Eq.~\eqref{eq1} with respect to~\eqref{eq0}.

Moreover, from the technical point of view, the proofs of the main results in the works, such as \cite{BSTW17, HS24, HT21, CIT21, CIT22, CITT24} dedicated to the analysis of~\eqref{eq0} or even more general models (such as fractional diffusion) which are still semilinear, are all based on employing Duhamel's formula (or similar representation formulas) as a starting points for estimates. Such a formula is not available for quasilinear diffusion, and thus, we have to look for different techniques in order to deal with Eq. \eqref{eq1}. As precedents for the analysis of solutions to Eq. \eqref{eq1}, the Fujita-type exponent
\begin{equation}
	p_F := m + \frac{\sigma+2}{N}>m. \label{FujitaExp}
\end{equation}
is identified in \cite{Qi98}, where it is proved that, for $m<p\leq p_F$, any non-trivial, non-negative solution to Eq.~\eqref{cp} with exponents as in~\eqref{exp} blows up in finite time, while for $p>p_F$, there is at least one non-negative initial condition for which the corresponding solution to~\eqref{cp} is global. The blow-up rates of solutions to Eq.~\eqref{eq1} with $1<m<p<p_F$ are obtained in \cite[Theorem 1.2]{AT05}. More results, such as local existence, uniqueness, initial traces, and also, in the range $p>p_F$, the threshold behavior of $u_0(x)$ as $|x|\to\infty$ separating finite time blow-up from global existence of solutions (which is called the second critical exponent) are available for equations of the form
\begin{equation*}
	\partial_t u =\Delta u^m+ K(x)u^p, \qquad t>0, \ x\in\real^N,
\end{equation*}
with weights $K(x)$ decaying as $|x|^{\sigma}$ as $|x|\to\infty$ but regularized at $x=0$, in works such as \cite{AdB91, Su02}. However, the singularity of $|x|^{\sigma}$ at $x=0$ is a significant feature which produces different results than in the case of a regularized form of it, while also making the analysis more complex. Thus, up to our knowledge, there are no works dealing with the analysis of Eq.~\eqref{eq1} either in the range $1<p\leq m$ or in the range $p>p_F$.

This is why the present paper is aimed at establishing a general, unified mathematical theory for Eq.~\eqref{eq1} with $m$, $p$, $\sigma$ as in~\eqref{exp}, including results concerning local existence of solutions, uniqueness and comparison principle, finite time blow-up or global existence of solutions in the ranges and for the initial conditions for which these properties hold true. We are now in a position to state and discuss our main results.

\section{Main results}\label{sec.mr}

We split the presentation into several paragraphs for the reader's convenience. From now on, for $t>0$ we write throughout the paper $u(t)$ for the mapping $x\mapsto u(t,x)$, $x\in\real^N$, and the space $L_{+}^r(\mathbb{R}^N)$ denotes the subspace of non-negative functions belonging to $L^r(\mathbb{R}^N)$.

\medskip

\noindent \textbf{A. Existence}. Owing to the singularity of $|x|^{\sigma}$ at $x=0$, the problem of existence is not a trivial one, as noticed, for example, in the linear case $m=p=1$, $\sigma=-2$ of Eq.~\eqref{eq1} in \cite{BG84}. However, in our range of exponents~\eqref{exp}, the local (in time) existence of weak solutions is always granted. Before stating the existence result, we recall the definition of a weak solution to~\eqref{cp}.

\begin{definition}\label{def.ws}
Let $r\ge 0$, $T\in (0,\infty]$ and $u_0\in L_+^{r+1}(\mathbb{R}^N)\cap L^\infty(\mathbb{R}^N)$. A non-negative weak solution to~\eqref{cp} on $[0,T)$ is a function
\begin{equation}
	u\in L_{\mathrm{loc}}^\infty\big([0,T),L_+^{r+1}(\mathbb{R}^N)\cap L^\infty(\mathbb{R}^N)\big)  \;\text{ with }\; \nabla u^{(m+r)/2} \in L_{\mathrm{loc}}^2\big( [0,T),L^2(\mathbb{R}^N) \big) \label{regws1}
\end{equation}
such that
\begin{equation}
	\int_0^T \int_{\mathbb{R}^N} \left[ (u_0-u) \partial_t\vartheta - u^m \Delta\vartheta - |x|^\sigma u^p \vartheta \right]\, dxds = 0 \label{wf1}
\end{equation}
for all $\vartheta\in C_c^2\big([0,T)\times\mathbb{R}^N\big)$.
\end{definition}

Note that all terms in~\eqref{wf1} make sense, according to~\eqref{regws1} and the assumption~\eqref{exp} on the parameters $m$, $p$ and $\sigma$. Indeed, observing that
\begin{equation*}
\begin{aligned}
	|x|^\sigma u^p(t,x) & \le \|u(t)\|_\infty^p |x|^\sigma, & \qquad (t,x)& \in (0,T)\times B(0,1), \\
	|x|^\sigma u^p(t,x) & \le \|u(t)\|_\infty^p, & \qquad (t,x)&\in (0,T)\times \big[ \mathbb{R}^N\setminus B(0,1) \big],
\end{aligned}
\end{equation*}
we see that the three terms involved in~\eqref{wf1} are integrable on $(0,T)\times\mathbb {R}^N$ due to~\eqref{exp}, \eqref{regws1} and the compactness of the support of $\vartheta$.

\begin{theorem}\label{th.exist}
Let $m$, $p$, $\sigma$ as in~\eqref{exp} and $u_0\in L_+^{r_1+1}(\mathbb{R}^N)\cap L^{\infty}(\mathbb{R}^N)$ for some
\begin{equation}\label{cond.exist}
r_1+1>r_c + 1 := \max\left\{1,\frac{N(p-m)}{\sigma+2}\right\}.
\end{equation}
Then, there exist $T_{\infty}\in(0,\infty]$ and a weak solution $u$ to the Cauchy problem~\eqref{cp} on $[0,T_\infty)$ in the sense of Definition~\ref{def.ws} with $r=r_1$. Moreover, $T_\infty=\infty$ when $p\in (1,p_G]$ and $T_\infty$ only depends on $N$, $m$, $p$, $\sigma$, $r_1$ and $\|u_0\|_{r_1+1}$ when $T_\infty<\infty$.

Assume further that $u_0\in L^q(\mathbb{R}^N)$ for some $q\in [1,r_1+1)$. Then
\begin{equation*}
	u \in L_{\mathrm{loc}}^\infty\big([0,T_\infty),L^{q}(\mathbb{R}^N)\big) \;\text{ with }\; \nabla u^{(m+q-1)/2} \in L_{\mathrm{loc}}^2\big( [0,T_\infty),L^2(\mathbb{R}^N) \big).
\end{equation*}
\end{theorem}

We notice that, setting
\begin{equation}\label{rexp}
	r_0:= \frac{N}{\sigma+2}(p-m)-1 = \frac{N}{\sigma+2}(p - p_F),
\end{equation}
we have $r_0\leq 0$ for $1<p\leq p_F$, so that \eqref{cond.exist} becomes $r_1 >r_c=1$ for $1<p\leq p_F$. Moreover, $r_0+1$ coincides with the critical exponent identified in the semilinear case $m=1$ in \cite{BSTW17} (denoted therein by $q_c$). The proof of Theorem~\ref{th.exist} follows from an approximation scheme, with the aid of some sharp \textit{a priori estimates}, employing an extended version of the Caffarelli-Kohn-Nirenberg inequalities \cite[Theorem]{CKN1984} (CKN for short) due to \cite{LiYan2023} as the main technical tool in their deduction. Precisely these sharp \textit{a priori} bounds, together with some regularizing effect, provide directly the global existence as a particular case of the general analysis when the condition $p\in (1,p_G]$ is in force, we refer the reader to Section~\ref{sec.ex} for the proof.

Before moving to the uniqueness issue, we point out that the weak solution constructed in Theorem~\ref{th.exist} with an initial condition $u_0\in L_+^1(\mathbb{R}^N)\cap L^\infty(\mathbb{R}^N)$ satisfies an improved version of~\eqref{wf1}.

\begin{corollary}\label{cor.exist}
Let $m$, $p$, $\sigma$ as in~\eqref{exp} and $u_0\in L_+^{1}(\mathbb{R}^N)\cap L^{\infty}(\mathbb{R}^N)$. Then the weak solution $u$ to~\eqref{cp} on $[0,T_\infty)$ constructed in Theorem~\ref{th.exist} satisfies
\begin{equation}
	u\in L_{\mathrm{loc}}^\infty\big([0,T_\infty),L_+^{1}(\mathbb{R}^N)\cap L^\infty(\mathbb{R}^N)\big)  \;\text{ with }\; \nabla u^{m} \in  L_{\mathrm{loc}}^2\big( [0,T_\infty),L^2(\mathbb{R}^N) \big) \label{regws2}
\end{equation}
and
\begin{equation}
	\int_0^{T_\infty} \int_{\mathbb{R}^N} \left[ (u_0-u) \partial_t\vartheta + \nabla u^m \cdot \nabla\vartheta - |x|^\sigma u^p \vartheta \right]\, dxds = 0 \label{wf2}
\end{equation}
for all $\vartheta\in C_c^1\big([0,T_\infty)\times\mathbb{R}^N\big)$.
\end{corollary}

The main difference between the weak formulations~\eqref{wf1} and~\eqref{wf2} lies in the handling of the diffusive term, which involves only one derivative of $\vartheta$ in~\eqref{wf2} due to the $L^2$-regularity of $\nabla u^m$ in~\eqref{regws2}. Observe that such a property is not included in~\eqref{regws1} when $r>m$. As we shall see below, the weak formulation~\eqref{wf2} is better suited for the study of uniqueness. Indeed, in contrast to the semilinear case $m=1$ for which the diffusion operator is a contraction in any $L^r$-space, there is less flexibility to deal with the uniqueness issue in the quasilinear case $m>1$ and so-called  \textsl{energy solutions}, as constructed in Corollary~\ref{cor.exist}, are the natural framework here \cite{AL79, Ca99, Ot96}.

\begin{remark}\label{rem.ssws}
	We point out here that the already mentioned self-similar solutions $\mathcal{U}_*$ and $\mathcal{U}^*$ to~\eqref{eq1} described in~\eqref{forward} and~\eqref{exponential} and constructed in \cite{IMS23} for $p\in (1,p_G)$ and in \cite{ILS24b} for $p=p_G$, respectively, are weak solutions to~\eqref{eq1} in the sense of Definition~\ref{def.ws} and Corollary~\ref{cor.exist}, even though the gradient of their profiles $f_*(|x|)$ and $f^*(|x|)$ behave like $|x|^{\sigma+1}$ as $x\to 0$ and thus exhibit a singularity when $N\ge 2$ and $\sigma\in (-2,-1)$, see \cite[Theorem~1.1]{IMS23} and \cite[Theorem~1.1]{ILS24b}. Still, $\nabla\mathcal{U}_*$ and $\nabla\mathcal{U}^*$ both satisfy~\eqref{regws2}, and $\mathcal{U}_*$ and $\mathcal{U}^*$  are weak solutions to~\eqref{eq1} in the sense of~\eqref{wf2}. In particular, for $p\in (1,p_G)$, $\mathcal{U}_*$ is a weak solution to~\eqref{cp} with initial condition $\mathcal{U}_*(0)\equiv 0$.
\end{remark}

\bigskip

\noindent \textbf{B. Uniqueness and comparison principle}. This is one of the most interesting results of this work, since we stress once more that, according to the global radially symmetric self-similar solution with initial trace equal to zero constructed in \cite{IMS23}, at least when $p\in (1,p_G)$ we are facing an example of strong non-uniqueness of solutions. Indeed, according to Remark~\ref{rem.ssws}, the Cauchy problem~\eqref{cp} with initial condition $u_0\equiv 0$ has two different solutions $u\equiv 0$ and $u=\mathcal{U}_*$ when $p\in (1,p_G)$. However, as we see in the next theorem, uniqueness and comparison principle hold true unconditionally when $p\ge p_G$ and $N\ge 3$ but require an extra condition of positivity of the initial condition $u_0$ near the origin when $p\in (1,p_G)$ and $N\ge 3$. A similar result is available in low space dimensions $N\in\{1,2\}$ but for a slightly different range of the parameter $p$.

\begin{theorem}\label{th.uniq}
Let $m$, $p$ and $\sigma$ be as in~\eqref{exp} and consider two functions $u_{0,i}\in  L_+^{1}(\mathbb{R}^N)\cap L^{\infty}(\mathbb{R}^N)$, $i\in\{1,2\}$, such that $u_{0,1}(x)\leq u_{0,2}(x)$ for any $x\in\mathbb{R}^N$. Let $T>0$ and $u_i$ be a weak solution to~\eqref{cp} on $(0,T)$ with initial condition $u_{0,i}$, $i\in\{1,2\}$, provided by Corollary~\ref{cor.exist}.

(a) If $p\ge p_G$ and $N\ge 3$, then we have $u_1(t,x)\leq u_2(t,x)$ for any $(t,x)\in(0,T)\times\mathbb{R}^N$. The same result holds true for $p>p_G$ when $N=2$ and  for $p>1-\sigma(m-1)$ when $N=1$.

(b) If $p\in (1,p_G)$,  $N\ge 3$ and there is $\delta>0$ such that $u_{0,i}>0$ on $B(0,\delta)$, $i\in\{1,2\}$, then we have $u_1(t,x)\leq u_2(t,x)$, for any $(t,x)\in(0,T)\times\mathbb{R}^N$. The same result holds true for $p\in (1,p_G]$ when $N=2$ and for $p\in (1,1-\sigma(m-1)]$ when $N=1$.
\end{theorem}

\begin{remark} \label{rem.uniq}
	It is interesting to notice that unconditional uniqueness holds true for a broader range of the exponent $p>1$ when $N\ge 3$ and we do not know yet whether the additional restrictions in dimension $N\in\{1,2\}$ are just a technical limitation of the proof or a true peculiarity.
\end{remark}

It is a well known fact that the comparison principle is one of the most important tools in developing the theory of nonlinear parabolic equations, and we shall see this principle in action in some of the proofs of the next results. In particular, we report the following well-posedness result.

\begin{corollary}\label{cor.maxwp}
	Let $u_0\in L_+^1(\mathbb{R}^N)\cap L^\infty(\mathbb{R}^N)$ and $p>p_G$. There are $T_{\mathrm{max}}(u_0)\in (0,\infty]$ and a unique non-negative weak solution $u$ to~\eqref{cp} on $[0,T_{\mathrm{max}}(u_0))$. In addition, if $T_{\mathrm{max}}(u_0)<\infty$, then
	\begin{equation*}
		\limsup_{t\nearrow T_{\mathrm{max}}(u_0)} \|u(t)\|_\infty = \infty.
	\end{equation*}
\end{corollary}

\bigskip

\noindent \textbf{C. Finite time blow-up}. As usual when dealing with an equation involving a source term, the question of whether finite time blow-up occurs and for which classes of initial conditions, arises naturally. As we have seen in Theorem~\ref{th.exist}, global existence of solutions is ensured when $p\in (1,p_G]$. Moreover, the availability of self-similar solutions in the form~\eqref{backward} established in \cite{ILS24a} and the results in \cite{Qi98} suggest that finite time blow-up is expected when $p\in (p_G,p_F]$, not necessarily requiring $p>m$ as in \cite{Qi98}. This is in fact the outcome of the next result.

\begin{theorem}\label{th.blowup}
Let $m$, $p$ and $\sigma$ be as in~\eqref{exp} and such that $p_G<p\leq p_F$. Consider $u_0\in L_+^1(\mathbb{R}^N)\cap L^\infty(\mathbb{R}^N)$ and let $u$ be the unique non-negative solution to the Cauchy problem~\eqref{cp} on $[0,T_{\mathrm{max}}(u_0))$ provided by Corollary~\ref{cor.maxwp}. Then $T_{\mathrm{max}}(u_0))<\infty$.
\end{theorem}

In order to prove Theorem~\ref{th.blowup}, we mix a number of different techniques, according to the ranges of exponents. More precisely, for $p_G<p<m$, finite time blow-up follows by comparison with suitable subsolutions constructed by mimicking the self-similar form (and thus taking strongly into account the positivity of the self-similarity exponents $\alpha$ and $\beta$ as in~\eqref{backward}), in an analogous fashion as in \cite{SGKM1995}. We do not provide a proof for the range $m<p\leq p_F$, as Theorem~\ref{th.blowup} is already proved in \cite[Theorem~1.6]{Qi98} in that case. We nevertheless point out that a slightly simpler proof can be performed by a scaling argument as in \cite[Theorem~26.1]{QuSo2019}, proceeding along the lines of \cite[Theorem~18.1(i)]{MiPo2001}. Finally, for the case $p=m$, a technique similar to Kaplan's method \cite{Ka1963}, employing as test function a particular solution obtained by variational techniques in \cite{IL}, is the main argument of the proof.

A natural question arises with respect to the behavior of solutions when $p>p_F$. In this range, similarly as in the spatially homogeneous case $\sigma=0$, solutions may either blow up in finite time or remain bounded for any $t\in(0,\infty)$, in dependence on the initial condition. To this end, let us introduce the energy
\begin{equation}\label{energy}
E(v):=\frac{\|\nabla v^m\|_2^2}{2}-\frac{m}{m+p}\int_{\mathbb{R}^N}|x|^{\sigma}v^{m+p}(x)\,dx,
\end{equation}
which is well-defined and finite for any $v\in L_+^{m+p}(\mathbb{R}^N,|x|^\sigma\,dx)$ such that $\nabla v^m\in L^2(\mathbb{R}^N)$. We prove that solutions emanating from an initial condition having negative energy have to blow up in finite time, whatever the value of $p\ge m$.

\begin{theorem}\label{th.blowup2}
Assume that $p\ge m$ and consider $u_0\in L^{1}(\mathbb{R}^N)\cap L_{+}^{\infty}(\mathbb{R}^N)$ such that $u_0\in L_+^{m+p}(\mathbb{R}^N,|x|^\sigma\,dx)$ and $\nabla u_0^m\in L^2(\mathbb{R}^N)$. If
	\begin{equation}
		E(u_0) = \frac{\|\nabla u_0^m\|_2^2}{2} - \frac{m}{m+p}\int_{\mathbb{R}^N} |x|^\sigma u_0(x)^{m+p}\ dx < 0, \label{bu6}
	\end{equation}
then the solution $u$ to the Cauchy problem~\eqref{cp} provided by Corollary~\ref{cor.maxwp} satisfies $T_{\mathrm{max}}(u_0)<\infty$ and thus blows up in a finite time. Moreover, for $t\in [0,T_{\mathrm{max}}(u_0))$,
	\begin{equation}
		\|u(t)\|_{m+1} \le \left[ \frac{m}{(p-1)(m+p)} \frac{\|u_0\|_{m+1}^{m+p}}{|E(u_0)|} \right]^{1/(p-1)} \left[ T_{\mathrm{max}}(u_0) - t \right]^{-1/(p-1)}.\label{bu7}
	\end{equation}
\end{theorem}

Notice that, although the finite time blow-up in the range $m\leq p\leq p_F$ is already stated in Theorem~\ref{th.blowup}, the estimate~\eqref{bu7} is new also in that range. When $m=1$ and $\sigma=0$, an analogous result can be found in \cite[Theorem~17.6]{QuSo2019} but without an upper bound at the blow-up time. In fact, the proof of Theorem~\ref{th.blowup2} strongly departs from the one in \cite{QuSo2019} and is performed in Section~\ref{sec.bu4}.

\bigskip

\noindent \textbf{D. Global existence of solutions for small data in a critical norm}. In the range $p>p_F$, it is expected that the Cauchy problem~\eqref{cp} may have solutions blowing up in finite time and solutions that are global in time, depending on the properties of the initial condition. It is thus an interesting problem to establish classes of initial conditions ensuring global existence. This is the aim of the following result, which is sharp with respect to the norm. Recall that the exponent $r_0$ has been defined in~\eqref{rexp}.

\begin{theorem}\label{th.global}
Let $m$, $p$ and $\sigma$ be as in~\eqref{exp} such that $p>p_F$. Consider an initial condition $u_0$ such that $u_0\in L_+^{r_0+1}(\mathbb{R}^N)\cap L^\infty(\mathbb{R}^N)$. Then, there is $C_0>0$ depending only on $N$, $m$, $p$ and $\sigma$ such that, if $\|u_0\|_{r_0+1}\le C_0$, then the solution $u$ to the Cauchy problem~\eqref{cp} is global in time. Moreover, $u\in L^\infty((0,\infty),L^{r}(\mathbb{R}^N))$ for all $r\in [r_0+1,\infty)$.
\end{theorem}

The proof of Theorem~\ref{th.global}, which also relies on a variant of the CKN inequalities, is given in Section~\ref{sec.gs}.

\bigskip

\noindent \textbf{E. Unboundedness and temporal grow-up rates for $p\in (1,p_G]$}. We finally prove that, despite the non-occurrence of finite time blow-up when $p\in (1,p_G]$, there is no non-trivial bounded solution to~\eqref{cp} in that case, a property which contrasts markedly with the bounded global solutions constructed in Theorem~\ref{th.global} when $p>p_F$.

\begin{theorem}\label{th.unbdd}
Let $m$, $p$ and $\sigma$ be as in~\eqref{exp} such that $p\in (1,p_G]$. Consider an initial condition $u_0$ such that $u_0\in L_+^{1}(\mathbb{R}^N)\cap L^\infty(\mathbb{R}^N)$, $u_0\not\equiv 0$. Then for any solution $u$ to the Cauchy problem~\eqref{cp} we have
\begin{equation}
	\lim_{t\to\infty} \|u(t)\|_\infty = \infty. \label{x01}
\end{equation}
More precisely, there is $c_\infty>0$ depending on $N$, $p$, $m$, $\sigma$ and $u_0$ such that
\begin{equation}
\begin{split}
	\|u(t)\|_\infty & \ge c_\infty t^{\alpha_*}, \qquad t\ge 0, \;\ p\in (1,p_G), \\
	\|u(t)\|_\infty & \ge c_\infty e^{\alpha^* t}, \qquad t\ge 0, \;\ p=p_G,
\end{split}\label{lowbd}
\end{equation}
where the exponents $\alpha_*>0$ and $\alpha^*>0$ are defined in~\eqref{forward} and~\eqref{exponential}, respectively.

Furthermore, if the initial condition $u_0$ is compactly supported, then there is $C_\infty>0$ depending on $N$, $p$, $m$, $\sigma$ and $u_0$ such that
\begin{equation}
\begin{split}
	\|u(t)\|_\infty & \le C_\infty t^{\alpha_*}, \qquad t\ge 0, \;\ p\in (1,p_G), \\
	\|u(t)\|_\infty & \le C_\infty e^{\alpha^* t}, \qquad t\ge 0, \;\ p=p_G,
\end{split}\label{upbd}
\end{equation}
and, when $p\in(1,p_G)$, the following convergence to self-similarity holds true:
\begin{equation}
	\lim_{t\to\infty} t^{-\alpha_*} \|u(t) -  \mathcal{U}_*(t)\|_\infty = 0, \label{x02}
\end{equation}
where $\mathcal{U}_*$ is the unique self-similar solution to~\eqref{eq1} of the form~\eqref{forward}, see \cite{IMS23}.
\end{theorem}

The proof of Theorem~\ref{th.unbdd} relies on comparison arguments and is performed in Section~\ref{sec.ub}, ending the paper.

\bigskip

For further reference, we gather here the definition of the main parameters used throughout the paper:
\begin{equation}
\begin{aligned}
	& p_G = 1 - \frac{\sigma(m-1)}{2}\in (1,m), & \qquad p_F = m + \frac{\sigma+2}{N} > m, \\
	& r_0 = \frac{N}{\sigma+2}(p-m)-1 = \frac{N}{\sigma+2}(p-p_F), & \qquad r_c = \max\{r_0,0\}\ge 0, \\
	& \omega_N := |B(0,1)|. &
\end{aligned}\label{x07}
\end{equation}

\section{Well-posedness}\label{sec.wp}

This section is dedicated to the proofs of Theorems~\ref{th.exist} and~\ref{th.uniq} and their consequences reported in Corollaries~\ref{cor.exist} and~\ref{cor.maxwp}, which together establish the well-posedness of the Cauchy problem~\eqref{cp}. We thus split this section into two subsections for the reader's convenience.

\subsection{Existence}\label{sec.ex}

The proof of Theorem~\ref{th.exist} follows from combining an approximation scheme with some sharp a priori bounds for solutions (that are also fulfilled by the approximating solutions). The building blocks of these a priori estimates are the celebrated CKN inequalities \cite[Theorem]{CKN1984}, but since in some ranges of exponents, estimates involving $L^q$-pseudo norms for $q\in(0,1)$ show up, we need an extended version of the CKN inequalities derived in \cite[Theorem 1.2]{LiYan2023} coping with such pseudo norms. We recall below, for the reader's convenience, the CKN inequalities we employ throughout this paper. Let $(q_1,q_2,q_3,\gamma_1,\gamma_2,\gamma_3,a)\in\mathbb{R}^7$ be such that
\begin{subequations}\label{CKN.cond}
\begin{equation}\label{CKN.cond1}
q_1>0, \quad q_2\geq1, \quad q_3>0, \quad a\in[0,1],
\end{equation}
\begin{equation}\label{CKN.cond2}
\frac{1}{q_1}+\frac{\gamma_1}{N}>0, \quad \frac{1}{q_2}+\frac{\gamma_2}{N}>0, \quad \frac{1}{q_3}+\frac{\gamma_3}{N}>0,
\end{equation}
\begin{equation}\label{CKN.cond3}
\frac{1}{q_1}+\frac{\gamma_1}{N}=a\left(\frac{1}{q_2}+\frac{\gamma_2-1}{N}\right)+(1-a)\left(\frac{1}{q_3}+\frac{\gamma_3}{N}\right),
\end{equation}
\begin{equation}\label{CKN.cond4}
\gamma_1\leq a\gamma_2+(1-a)\gamma_3,
\end{equation}
\begin{equation}\label{CKN.cond5}
\frac{1}{q_1}\leq\frac{a}{q_2}+\frac{1-a}{q_3}, \quad {\rm if} \ a=0 \ {\rm or} \ a=1.
\end{equation}
\end{subequations}
We then have the following interpolation inequalities, see \cite[Theorem]{CKN1984} and \cite[Theorem~1.2]{LiYan2023}.

\begin{theorem}[CKN inequalities]\label{th.CKN}
Let $(q_1,q_2,q_3,\gamma_1,\gamma_2,\gamma_3,a)\in\mathbb{R}^7$ such that~\eqref{CKN.cond1} and~\eqref{CKN.cond2} are satisfied. Then, there is some positive constant $C$ such that
\begin{equation}\label{CKN}
\||x|^{\gamma_1} z\|_{q_1}\leq C\||x|^{\gamma_2}\nabla z\|_{q_2}^a\||x|^{\gamma_3}z\|_{q_3}^{1-a}
\end{equation}
holds true for all $z\in C_0^1(\mathbb{R}^N)$ if and only if~\eqref{CKN.cond3}, \eqref{CKN.cond4} and~\eqref{CKN.cond5} are satisfied. Moreover, in any compact set in the space of parameters satisfying~\eqref{CKN.cond1} and~\eqref{CKN.cond2}, the constant $C$ is bounded.
\end{theorem}

We next apply the CKN inequalities in the following particular case which will be an essential tool in the proof of Theorem~\ref{th.exist}.

\begin{proposition}\label{prop.ex1}
Let $r_1\in [r_c,\infty)$, $\bar{r}\in (r_1,\infty)$ and $r\in [r_1,\bar{r}]$. Consider $w\in L_+^{r_1+1}(\mathbb{R}^N)$ such that $w^{(m+r)/2}\in \dot{H}^1(\mathbb{R}^N)$. Then $w\in L^{p+r}(\mathbb{R}^N, |x|^\sigma dx)$ and there is $\Lambda(\bar{r})\ge 1$ depending only on $N$, $m$, $p$, $\sigma$ and $\bar{r}$ such that
\begin{equation}
	\int_{\mathbb{R}^N} |x|^\sigma w^{p+r}(x)\ dx \le \Lambda(\bar{r}) \left\|\nabla w^{(m+r)/2}\right\|_2^{2\omega_r}\left(\|w\|_{r_1+1}^{r_1+1}\right)^{\mu_r} \label{ex02}
\end{equation}
with
\begin{equation}\label{ex01}
\begin{aligned}
	\omega_r & := \frac{N(p-1)-\sigma(r_1+1)+N(r-r_1)}{N(m-1)+2(r_1+1)+N(r-r_1)}\in (0,1], \\
	\mu_r & := \frac{(N-2)(m-p)+(\sigma+2)(m+r)}{N(m-1)+2(r_1+1)+N(r-r_1)} >0.
\end{aligned}
\end{equation}
Moreover, $\omega_r=1$ if and only if $r_1=r_0$.
\end{proposition}

\begin{proof}
Let $r_1\in [r_c,\infty)$, $\bar{r}\in (r_1,\infty)$ and $r\in [r_1,\bar{r}]$, and consider $w\in L_+^{r+1}(\mathbb{R}^N)$ such that $W:= w^{(m+r)/2}\in \dot{H}^1(\mathbb{R}^N)$. Then
\begin{equation}
	\int_{\mathbb{R}^N} |x|^\sigma w^{p+r}(x)\ dx = \int_{\mathbb{R}^N} \left[ |x|^{\sigma(m+r)/[2(p+r)]} W(x) \right]^{2(p+r)/(m+r)}\ dx \label{ex03}
\end{equation}
and
\begin{equation}
	\|w\|_{r_1+1}^{r_1+1} = \int_{\mathbb{R}^N} W(x)^{2(r_1+1)/(m+r)}\ dx. \label{ex04}
\end{equation}
We next apply Theorem~\ref{th.CKN} with the exponents
\begin{equation*}
	(q_1,q_2,q_3,\gamma_1,\gamma_2,\gamma_3,a) = \left(\frac{2(p+r)}{m+r},2,\frac{2(r_1+1)}{m+r},\frac{\sigma(m+r)}{2(p+r)},0,0,a \right)
\end{equation*}
and
\begin{equation*}
	a = \frac{(m+r)[N(p-1)-\sigma(r_1+1)+N(r-r_1)]}{(p+r)[N(m-1)+2(r_1+1)+N(r-r_1)]}.
\end{equation*}
We check below that the conditions~\eqref{CKN.cond} are fulfilled by the previous choice. It is obvious from~\eqref{exp} and the property $r\ge r_1$ that $a>0$ and we find after straightforward calculations that
\begin{equation}\label{ex06}
	1-a=\frac{(r_1+1)[(N-2)(m-p)+(\sigma+2)(m+r)]}{(p+r)[N(m-1)+2(r_1+1)+N(r-r_1)]}.
\end{equation}
Notice that, on the one hand, if $1<p\leq p_F$, then $r_c=0$ and, for $N\ge 2$,
\begin{equation*}
\begin{split}
	(N-2)(m-p)+(\sigma+2)(m+r)&\geq -\frac{(\sigma+2)(N-2)}{N} + (\sigma+2)(m+r)\\
	& \geq (\sigma+2) \left(m-\frac{N-2}{N}\right)>0,
\end{split}
\end{equation*}
while, for $N=1$,
\begin{equation*}
	\begin{split}
		(N-2)(m-p)+(\sigma+2)(m+r)&\geq p-m + m(\sigma+2)\\
		& \geq p + m(\sigma+1)>0.
	\end{split}
\end{equation*}
On the other hand, if $p>p_F>m$ and $r\geq r_1\geq r_c = r_0$, then
\begin{equation*}
\begin{split}
	(N-2)(m-p)+(\sigma+2)(m+r)&\geq (N-2)(m-p) + (\sigma+2) \left[ m-1+\frac{N(p-m)}{\sigma+2} \right] \\
	&=(\sigma+2)(m-1)+2(p-m)>0.
\end{split}
\end{equation*}
We deduce from~\eqref{ex06} and the previous estimates that indeed $a\in(0,1)$ and the condition~\eqref{CKN.cond1} is fulfilled, since the other inequalities it requires are obviously satisfied. With respect to the condition~\eqref{CKN.cond2}, the second and the third inequalities are obvious, while the first one follows from the restriction $\sigma+N>0$ in~\eqref{exp} and
\begin{equation*}
	\frac{1}{q_1}+\frac{\gamma_1}{N}=\frac{(\sigma+N)(m+r)}{2N(p+r)}>0.
\end{equation*}
The condition~\eqref{CKN.cond3} is equivalent to
\begin{equation*}
\frac{(\sigma+N)(m+r)}{2N(p+r)}=a\frac{N-2}{2N}+(1-a)\frac{m+r}{2(r_1+1)},
\end{equation*}
which follows by direct calculation, while the condition~\eqref{CKN.cond4} becomes $\gamma_1\leq0$, which follows from the negativity of $\sigma$. Finally, the condition~\eqref{CKN.cond5} is never in force, since $a\in(0,1)$, as shown above. Thus, we infer from~\eqref{CKN}, taking into account the formulas~\eqref{ex03} and~\eqref{ex04}, that there is a positive constant $\Lambda(\bar{r})\ge 1$ depending only on $N$, $m$, $p$, $\bar{r}$ and $\sigma$ such that
\begin{align*}
& \left( \int_{\mathbb{R}^N} \left[ |x|^{\sigma(m+r)/[2(p+r)]} W(x) \right]^{2(p+r)/(m+r)}\ dx \right)^{(m+r)/[2(p+r)]} \\
& \hspace{4cm} \le \Lambda(\bar{r})^{(m+r)/[2(p+r)]} \|\nabla W\|_2^{a} \|W\|_{2(r_1+1)/(m+r)}^{1-a},
\end{align*}
which readily leads to~\eqref{ex02} by noticing that
\begin{equation*}
	a=\frac{m+r}{p+r}\omega_r, \quad 1-a=\frac{r_1+1}{p+r}\mu_r.
\end{equation*}
Moreover, the constant $\Lambda(\bar{r})$ is indeed independent of $r_1$ and $r$, thanks to Theorem~\ref{th.CKN} and the bounds
\begin{equation*}
	\frac{2(r_1+1)}{m+r} \in \left[ \frac{2}{m+\bar{r}},2 \right] \;\;\text{ and }\;\; \frac{p+r}{m+r} \in\left[ \min\left\{ 1 , \frac{p}{m} \right\}, \max\left\{ 1 , \frac{p}{m} \right\} \right].
\end{equation*}
Finally, the positivity of $a$ and $1-a$ entails that of both $\omega_r$ and $\mu_r$, while the definitions~\eqref{x07} of $r_0$ and~\eqref{ex01} of $\omega_r$ give
\begin{equation*}
	1-\omega_r = \frac{N(m-p)+(\sigma+2)(r_1+1)}{N(m-1)+2(r_1+1)+N(r-r_1)} = \frac{(\sigma+2)(r_1-r_0)}{N(m-1)+2(r_1+1)+N(r-r_1)},
\end{equation*}
so that $\omega_r<1$ when $r_1>r_0$ and $\omega_{r}=1$ when $r_1=r_0$.
\end{proof}

We next apply Proposition~\ref{prop.ex1} in order to derive an \textit{a priori} estimate for the $L^{r+1}$-norm, $r\in (r_c,\infty)$, of a solution $u$ to Eq.~\eqref{eq1}.

\begin{proposition}\label{prop.ex2}
Let $m$, $p$, $\sigma$ be as in~\eqref{exp} and $r_1\in (r_c,\infty)$. Then there exist $T_\infty\in (0,\infty]$  depending only on $N$, $m$, $p$, $\sigma$, $r_1$ and $\|u_0\|_{r_1+1}$ such that, for any $r\in [r_1,\infty)$ and $T\in (0,T_\infty)$, there is a positive constant $C_1(T,r)$ depending on $N$, $m$, $p$, $\sigma$, $r_1$, $\|u_0\|_{r_1+1}$, $T$, $r$ and $\|u_0\|_{r+1}$ such that
\begin{equation}
	\|u(t)\|_{r+1}^{r+1} + \int_0^t  \|\nabla u^{(m+r)/2}(s)\|_2^2\,ds \le C_1(T,r), \qquad t\in [0,T]. \label{ex11}
\end{equation}
\end{proposition}

\begin{proof} Let $r_1\in (r_c,\infty)$.

\medskip

\noindent\textbf{Step~1. A differential inequality for $\|u\|_{r+1}$, $r\ge r_1$}.  For $r\in [r_1,\infty)$, we multiply Eq.~\eqref{eq1} by $u^r$ and integrate with respect to the space variable to find
\begin{equation*}
\begin{split}
\frac{1}{r+1}\frac{d}{dt}\|u(t)\|_{r+1}^{r+1}&=-r\int_{\mathbb{R}^N}(u^{r-1}\nabla u\cdot\nabla u^m)(t,x)\,dx+\int_{\mathbb{R}^N}|x|^{\sigma}u^{p+r}(t,x)\,dx\\
&=-mr\int_{\mathbb{R}^N}(u^{m+r-2}|\nabla u|^2)(t,x)\,dx+\int_{\mathbb{R}^N}|x|^{\sigma}u^{p+r}(t,x)\,dx,
\end{split}
\end{equation*}
or, equivalently,
\begin{equation}\label{ex09}
	\frac{d}{dt}\|u(t)\|_{r+1}^{r+1}+\frac{4mr(r+1)}{(m+r)^2}\|\nabla u^{(m+r)/2}(t)\|_{2}^2 = (r+1) \int_{\mathbb{R}^N}|x|^{\sigma}u^{p+r}(t,x)\,dx.
\end{equation}
We apply Proposition~\ref{prop.ex1} with $w=u(t)$ to estimate the right hand side of~\eqref{ex09} and infer from~\eqref{ex02} and Young's inequality that
\begin{align}
	(r+1) \int_{\mathbb{R}^N} |x|^{\sigma} u^{p+r}(t,x)\,dx & \leq (r+1) \Lambda(r) \|\nabla u^{(m+r)/2}(t)\|_2^{2\omega_r} \big(\|u\|_{r_1+1}^{r_1+1}\big)^{\mu_r} \nonumber\\
	& \leq \frac{2mr(r+1)}{(m+r)^2} \|\nabla u^{(m+r)/2}(t)\|_2^2 + C(r)\big(\|u\|_{r_1+1}^{r_1+1}\big)^{\mu_r/(1-\omega_r)},\label{ex10}
\end{align}
recalling that $\omega_r < 1$ due to $r_1>r_c\ge r_0$. Noticing that
\begin{equation}
	\nu_r := \frac{\mu_r}{1-\omega_r}=\frac{(N-2)(m-p)+(\sigma+2)(m+r)}{N(m-p)+(\sigma+2)(r_1+1)}>0, \label{ex08}
\end{equation}
we readily obtain from~\eqref{ex09} and~\eqref{ex10} that
\begin{equation}\label{ex07}
	\frac{d}{dt}\|u(t)\|_{r+1}^{r+1} + \frac{2mr(r+1)}{(m+r)^2} \|\nabla u^{(m+r)/2}(t)\|_2^2 \leq C(r) \big( \|u(t)\|_{r_1+1}^{r_1+1} \big)^{\nu_r}.
\end{equation}
Observe that
\begin{equation}
	1 - \nu_r = - \frac{2(p-p_G)+ (\sigma+2)(r-r_1)}{N(m-p)+(\sigma+2)(r_1+1)}. \label{x09}
\end{equation}

\medskip

\noindent\textbf{Step~2. A bound in $L^{r_1+1}(\mathbb{R}^N)$}. We first take $r=r_1$ in~\eqref{ex07} to obtain the differential inequality
\begin{equation}\label{x10}
	\frac{d}{dt}\|u(t)\|_{r_1+1}^{r_1+1} + \frac{2mr_1(r_1+1)}{(m+r_1)^2} \|\nabla u^{(m+r_1)/2}(t)\|_2^2 \leq C \big( \|u(t)\|_{r_1+1}^{r_1+1} \big)^{\nu_{r_1}},
\end{equation}
with
\begin{equation}
	1 - \nu_{r_1} = - \frac{2(p-p_G)}{N(m-p)+(\sigma+2)(r_1+1)}, \label{x08}
\end{equation}
according to~\eqref{x09}. At this point, the consequences drawn from the differential inequality~\eqref{x10} depend on the sign of $1-\nu_{r_1}$ and thus on the sign of $p-p_G$. We then split the analysis according to the sign of $p-p_G$.

\noindent $\bullet$ $p\in (1,p_G)$. In that case, $\nu_{r_1}<1$ and we infer from~\eqref{x10} that
\begin{equation*}
	\frac{d}{dt} \big( \|u(t)\|_{r_1+1}^{r_1+1} \big)^{1-\nu_{r_1}} \le C,
\end{equation*}
whence
\begin{equation*}
	\|u(t)\|_{r_1+1}^{r_1+1} \le \left( \|u_0\|_{r_1+1}^{(r_1+1)(1-\nu_{r_1})} + C t \right)^{1/(1-\nu_{r_1})}, \qquad t\ge 0.
\end{equation*}
Setting $T_\infty := \infty$, we deduce from~\eqref{x10} and the above estimate that, for any $T\in (0,T_\infty)$,
\begin{equation}
	\|u(t)\|_{r_1+1}^{r_1+1} + \int_0^t \|\nabla u^{(m+r_1)/2}(s)\|_2^2\,ds \le C(T), \qquad t\in [0,T]. \label{x11}
\end{equation}

\noindent $\bullet$ $p =p_G$. In that case, $\nu_{r_1}=1$ and we infer from~\eqref{x10} that
\begin{equation*}
	\|u(t)\|_{r_1+1}^{r_1+1} \le \|u_0\|_{r_1+1}^{r_1+1} e^{Ct}, \qquad t\ge 0.
\end{equation*}
Setting $T_\infty := \infty$, we deduce from~\eqref{x10} and the above estimate that~\eqref{x11} holds true as well in that case.

\noindent $\bullet$ $p > p_G$. In that case, $\nu_{r_1}>1$ and we infer from~\eqref{x10} that
\begin{equation*}
	\frac{d}{dt} \big( \|u(t)\|_{r_1+1}^{r_1+1} \big)^{1-\nu_{r_1}} \ge -C,
\end{equation*}
whence
\begin{equation*}
	\|u(t)\|_{r_1+1}^{r_1+1} \le \left( \|u_0\|_{r_1+1}^{(r_1+1)(1-\nu_{r_1})} - C t \right)^{1/(1-\nu_{r_1})}, \qquad t\in \left[ 0 , \frac{\|u_0\|_{r_1+1}^{(r_1+1)(1-\nu_{r_1})}}{C} \right).
\end{equation*}
Therefore, setting $T_\infty := \|u_0\|_{r_1+1}^{(r_1+1)(1-\nu_{r_1})}/C$, the bound~\eqref{x11} follows from~\eqref{x10} and the above estimate.

\medskip

\noindent\textbf{Step~3. A bound in $L^{r+1}(\mathbb{R}^N)$, $r\in (r_1,\infty)$}. Let $r\in (r_1,\infty)$ and $T\in (0,T_\infty)$, with $T_\infty$ defined in \textbf{Step~2} according to the value of $p-p_G$. We infer from~\eqref{ex07} and~\eqref{x11} that, for $t\in [0,T]$,
\begin{equation*}
	\frac{d}{dt}\|u(t)\|_{r+1}^{r+1} + \frac{2mr(r+1)}{(m+r)^2} \|\nabla u^{(m+r)/2}(t)\|_2^2 \leq C(r)\big( \|u(t)\|_{r_1+1}^{r_1+1} \big)^{\nu_{r}} \le C(T,r),
\end{equation*}
from which~\eqref{ex11} readily follows after integration with respect to time.
\end{proof}

The next result extends the local (or global) \textit{a priori} estimate to the space $L^{\infty}(\mathbb{R}^N)$.

\begin{proposition}\label{prop.ex3}
Let $m$, $p$ and $\sigma$ be as in~\eqref{exp} and $r_1\in (r_c,\infty)$. For any $T\in (0,T_\infty)$, there is a positive constant $C_2(T)$ depending only on $N$, $m$, $p$, $\sigma$, $r_1$, $\|u_0\|_{r_1+1}$ and $\|u_0\|_\infty$ such that
\begin{equation*}
	\|u(t)\|_\infty \le C_2(T), \qquad t\in [0,T].
\end{equation*}
\end{proposition}

\begin{proof}
We investigate the integrability of the source term $S(t,x) := |x|^\sigma u^p(t,x)$ defined for $(t,x)\in [0,T_\infty)\times\mathbb{R}^N$. Let $r\geq1$ and $(\alpha,\beta)\in (0,1)^2$ to be chosen later. For $T\in (0,T_\infty)$ and $t\in [0,T]$, it follows from H\"older's inequality that
\begin{equation*}
\begin{split}
\int_{\mathbb{R}^N}&S^r(t,x)\,dx=\int_{B(0,1)}\big[|x|^{\sigma}u^p(t,x)\big]^{r}\,dx
+\int_{\mathbb{R}^N\setminus B(0,1)}\big[|x|^{\sigma}u^p(t,x)\big]^{r}\,dx\\
&\leq\left(\int_{B(0,1)}|x|^{r\sigma/\alpha}\,dx\right)^{\alpha}\left(\int_{B(0,1)}u(t,x)^{pr/(1-\alpha)}\,dx\right)^{1-\alpha}\\
& \qquad +\left(\int_{\mathbb{R}^N\setminus B(0,1)}|x|^{r\sigma/\beta}\,dx\right)^{\beta}\left(\int_{\mathbb{R}^N\setminus B(0,1)}u(t,x)^{pr/(1-\beta)}\,dx\right)^{1-\beta}\\
&\leq(N\omega_{N})^{\alpha} \left(\frac{\alpha}{r\sigma+N\alpha}\right)^{\alpha} \|u\|^{pr}_{pr/(1-\alpha)}
+ (N\omega_{N})^{\beta} \left(\frac{\beta}{-r\sigma-N\beta}\right)^{\beta} \|u\|^{pr}_{pr/(1-\beta)},
\end{split}
\end{equation*}
and we infer from Proposition~\ref{prop.ex2} that the right hand side of the above inequality is finite and bounded for $t\in [0,T]$ provided
\begin{equation}\label{ex12}
r\sigma+N\alpha>0>r\sigma+N\beta, \quad \min\left\{\frac{pr}{1-\alpha},\frac{pr}{1-\beta}\right\}\geq r_1+1.
\end{equation}
On the one hand, the first condition in~\eqref{ex12} reads
\begin{equation}
	0<\beta<\frac{r|\sigma|}{N}<\alpha<1. \label{x12}
\end{equation}
On the other hand, taking into account that $\beta<\alpha$ by~\eqref{x12}, the second condition in~\eqref{ex12} amounts to
\begin{equation*}
	pr>(1-\beta)(r_1+1),
\end{equation*}
or, equivalently,
\begin{equation}
\beta > 1 - \frac{pr}{r_1+1}. \label{x13}
\end{equation}
We may thus find $0<\beta<\alpha<1$ satisfying~\eqref{x12} and~\eqref{x13} if and only if
\begin{equation*}
	1 > \frac{r|\sigma|}{N} > 1 - \frac{pr}{r_1+1} \;\;\text{ and }\;\; r\ge 1;
\end{equation*}
that is,
\begin{equation*}
	\frac{N(r_1+1)}{Np + (r_1+1)|\sigma|} < r < \frac{N}{|\sigma|}  \;\;\text{ with }\;\; r\ge 1,
\end{equation*}
a choice which is always possible, since $N>|\sigma|$ and
\begin{equation*}
	\frac{N}{|\sigma|} - \frac{N(r_1+1)}{Np + (r_1+1)|\sigma|} = \frac{N}{|\sigma|} \frac{Np}{Np + (r_1+1)|\sigma|} > 0.
\end{equation*}
We have thus established that
\begin{equation}
	\|S(t)\|_{r} \le C(T,r), \qquad t\in [0,T], \quad r\in \mathcal{I}_{r_1} := \left( \max\left\{ 1 , \frac{N(r_1+1)}{Np + (r_1+1)|\sigma|} \right\} , \frac{N}{|\sigma|} \right). \label{x14}
\end{equation}

Now, pick $r\in \mathcal{I}_{r_1}\cap (N/2,N/|\sigma|)$ if $N\geq2$ or $r\in \mathcal{I}_{r_1}\cap (1,N/|\sigma|)$ if $N=1$. We infer from~\eqref{x14} and \cite[Th\'eor\`eme~1]{Ka93} (with $s=r$, $p=q=\infty$, $\epsilon=1$ and $C=C_2$ in the notation therein) that
\begin{equation*}
\begin{split}
	\|u(t)\|_{\infty}&\leq 2+C\|u_0\|_{\infty} + C(mt)^{-N/2r}\int_0^{t/2}\|S(\tau)\|_r\,d\tau \\
	& \qquad + Cm^{-N/(2r)}\int_0^{t/2}\tau^{-N/2r}\|S(t-\tau)\|_{r}\,d\tau\\
	&\leq 2+C\|u_0\|_{\infty} + C(T,r) t^{(2r-N)/2r} + \frac{2rC(T,r)}{2r-N} t^{(2r-N)/2r},
\end{split}
\end{equation*}
thanks to the choice of $r$ which guarantees that $N/2r<1$. Consequently,
\begin{equation}\label{ex15}
	\|u(t)\|_{\infty}\leq C(T,r)t^{(2r-N)/2r}\leq C(T,r), \quad t\in[0,T],
\end{equation}
which completes the proof.
\end{proof}

We are now in a position to complete the proof of the local (and global, if $p\leq p_G$) existence of solutions to the Cauchy problem~\eqref{cp}.

\begin{proof}[Proof of Theorem~\ref{th.exist}]
Let $r_1>\max\{0,r_c\}$ with $r_c$ defined in Theorem~\ref{th.exist}, see also~\eqref{x07}, and $u_0\in L_+^{r_1+1}(\mathbb{R}^N)\cap L^\infty(\mathbb{R}^N)$. Given $\eta\in(0,1)$, let 
\begin{equation*}
	u_{\eta} \in C\big([0,\infty),L_+^1(\mathbb{R}^N)\big) \cap L_{\mathrm{loc}}^\infty\big([0,\infty), L^\infty(\mathbb{R}^N)\big)
\end{equation*}
be the unique global weak solution to
\begin{subequations}\label{cp.eta}
\begin{align}
	\partial_t u_{\eta} & = \Delta u_{\eta}^m + \big(|x|^2+\eta^2\big)^{\sigma/2} \frac{u_{\eta}^p}{1+\eta u_{\eta}^{p-1}},  \qquad t>0, \ x\in\real^N, \label{eq.eta}\\
	u_\eta(0) & = u_{0,\eta} := u_0 \mathbf{1}_{B(0,1/\eta)}, \qquad \ x\in\real^N. \label{ic.eta}
\end{align}
\end{subequations}
Since $u_{0,\eta}\in L_+^1(\mathbb{R}^N)\cap L^\infty(\mathbb{R}^N)$, the well-posedness of~\eqref{cp.eta} follows from classical arguments, after noticing that
\begin{equation*}
	\left| \big(|x|^2+\eta^2\big)^{\sigma/2} \frac{X^p}{1+\eta X^{p-1}} - \big(|x|^2+\eta^2\big)^{\sigma/2} \frac{Y^p}{1+\eta Y^{p-1}} \right| \le \eta^{\sigma-1} |X-Y|, \qquad (X,Y)\in [0,\infty)^2,
\end{equation*}
so that Eq.~\eqref{eq.eta} is a Lipschitz perturbation of the porous medium equation. We furthermore easily notice from the negativity of $\sigma$ that, if $0<\eta_1<\eta_2<1$, then $u_{\eta_2}$ is a subsolution to~\eqref{eq.eta} with $\eta=\eta_1$. Since $u_{0,\eta_2}\le u_{0,\eta_1}$ in $\mathbb{R}^N$, the comparison principle applied to~\eqref{eq.eta} entails that
\begin{equation}\label{ex17}
u_{\eta_1}(t,x)\geq u_{\eta_2}(t,x), \qquad (t,x)\in(0,\infty)\times\mathbb{R}^N.
\end{equation}
Moreover, by multiplying~\eqref{eq.eta} by $u_{\eta}^r$, $r\in [r_1,\infty)$, and taking into account that $\|u_{0,\eta}\|_q \le \|u_0\|_q$ for all $q\in [r_1,\infty]$, which implies in particular that
\begin{equation*}
	\|u_0\|_{r_1+1}^{(r_1+1)(1-\nu_{r_1})} \le \|u_{0,\eta}\|_{r_1+1}^{(r_1+1)(1-\nu_{r_1})}, \qquad \eta\in (0,1),
\end{equation*}
when $\nu_{r_1}>1$, and
\begin{equation}\label{ex16}
	\big(|x|^2+\eta^2\big)^{\sigma/2}\frac{u_{\eta}^p}{1+\eta u_{\eta}^{p-1}} \leq |x|^{\sigma}u_{\eta}^p \;\;\text{ in }\;\; (0,\infty)\times\mathbb{R}^N,
\end{equation}
we can repeat the lines of the proofs of Propositions~\ref{prop.ex2} and~\ref{prop.ex3} to conclude that there exists $T_\infty\in (0,\infty]$ depending only on $N$, $m$, $p$, $\sigma$, $r_1$ and $\|u_0\|_{r_1+1}$ such that, for any $r\in [r_1,\infty]$, $T\in (0,T_\infty)$ and $\eta\in (0,1)$, there is a positive constant $C_1(T,r)$ depending on $N$, $m$, $p$, $\sigma$, $r_1$, $\|u_0\|_{r_1+1}$, $T$, $r$ and $\|u_0\|_{r+1}$ such that
\begin{equation}
	\|u_\eta(t)\|_{r+1}^{r+1} + \int_0^t  \|\nabla u_\eta^{(m+r)/2}(s)\|_2^2\,ds \le C_1(T,r), \qquad t\in [0,T]. \label{y01}
\end{equation}
We thus introduce
\begin{equation}
	u(t,x) := \lim\limits_{\eta\to0} u_{\eta}(t,x) = \sup_{\eta\in (0,1)} u_\eta(t,x), \quad (t,x)\in [0,T_\infty)\times\mathbb{R}^N, \label{y02}
\end{equation}
which is well defined by the monotonicity in~\eqref{ex17} and the time-dependent bounds~\eqref{y01}. Classical arguments then allow us to deduce from~\eqref{y01} and~\eqref{y02} that $u$ is a weak solution to~\eqref{cp} on $[0,T_\infty)$ in the sense of Definition~\ref{def.ws}.

Assume further that $u_0\in L^q(\mathbb{R}^N)$ for some $q\in [1,r_1+1)$ and consider $T\in (0,T_\infty)$. It follows from~\eqref{eq.eta} and~\eqref{ex16} that, for $\eta\in (0,1)$ and $t\in (0,T)$,
\begin{align*}
	\frac{1}{q} \frac{d}{dt} \|u_\eta\|_q^q & \le - \frac{4m(q-1)}{(m+q-1)^2} \|\nabla u_\eta^{(m+q-1)/2}\|_2^2 + \int_{B(0,1)} |x|^\sigma u_\eta^{p+q-1}\,dx + \int_{\mathbb{R}^N \setminus B(0,1)}|x|^\sigma u_\eta^{p+q-1}\, dx \\
	& \le - \frac{4m(q-1)}{(m+q-1)^2} \|\nabla u_\eta^{(m+q-1)/2}\|_2^2 + \frac{N\omega_N}{\sigma+N} \|u_\eta\|_\infty^{p+q-1} + \|u_\eta\|_\infty^{p-1}\int_{\mathbb{R}^N} u_\eta^{q}\, dx.
\end{align*}
Hence, thanks to~\eqref{y01} (with $r=\infty$),
\begin{equation*}
	 \frac{d}{dt} \|u_\eta\|_q^q + \frac{4mq(q-1)}{(m+q-1)^2} \|\nabla u_\eta^{(m+q-1)/2}\|_2^2 \le \frac{N\omega_N}{\sigma+N} C_1(T,\infty)^{p+q-1} + C_1(T,\infty)^{p-1} \|u_\eta\|_q^{q}
\end{equation*}
for $t\in (0,T)$ and we infer from Gronwall's lemma that, for $t\in [0,T]$,
\begin{align*}
	 \|u_\eta(t)\|_q^q + \frac{4mq(q-1)}{(m+q-1)^2} \int_0^t \|\nabla u_\eta^{(m+q-1)/2}(s)\|_2^2\, ds & \le \left( \|u_{0,\eta}\|_q^q + \frac{N\omega_N}{\sigma+N} C_1(T,\infty)^q\right) e^{C_1(T,\infty)^{p-1}t} \\
	 & \le C_2(T,q) (1+\|u_0\|_q^q).
\end{align*}
Letting $\eta\to 0$ in the above inequality with the help of~\eqref{y02} completes the proof.
\end{proof}

\begin{proof}[Proof of Corollary~\ref{cor.exist}]
Since $u_0\in L_+^1(\mathbb{R}^N)\cap L^\infty(\mathbb{R}^N)$, the condition~\eqref{cond.exist} is satisfied for $r_1=r_c+1>r_c$ and the regularity property~\eqref{regws2} readily follows from~\eqref{regws1}	(with $r\in [r_c+1,\infty]$) and Theorem~\ref{th.exist} with $q=1$ and $q=m+1$. Combining the integrability of $\nabla u^m$ derived in~\eqref{regws2} with~\eqref{wf1} and a classical density argument gives~\eqref{wf2} and completes the proof.
\end{proof}

\subsection{Uniqueness and comparison principle}\label{sec.un}

In order to prove Theorem~\ref{th.uniq}, we first need a preparatory result, which is an extension of \cite[Lemma 1]{GP76}.

\begin{lemma}\label{lem.ineq}
Let $p\geq1$, $m\geq1$ and $\tau\in[0,2]$. Then, there exists a constant $C(m,p,\tau)>0$ such that
\begin{equation}\label{ineq}
|X^p-Y^p|^2 \leq C(m,p,\tau) \max\{X,Y\}^{2(p-1)-\tau(m-1)} |X^m-Y^m|^{\tau} |X-Y|^{2-\tau},
\end{equation}
for any $(X,Y)\in[0,\infty)^2$.
\end{lemma}

\begin{proof}
Since~\eqref{ineq} is obvious for $X=Y\in[0,\infty)$, we may assume without loss of generality that $X>Y$ and thus $X>0$. Then, setting $\zeta=Y/X\in[0,1)$, we find
\begin{equation}\label{interm1}
\frac{|X^p-Y^p|^2}{|X^m-Y^m|^{\tau} |X-Y|^{2-\tau}} = X^{2(p-1)-\tau(m-1)} \left(\frac{1-\zeta^p}{1-\zeta^m}\right)^{\tau}
\left(\frac{1-\zeta^p}{1-\zeta}\right)^{2-\tau}.
\end{equation}
Since
$$
\lim\limits_{\zeta\to1}\frac{1-\zeta^p}{1-\zeta^m}=\frac{p}{m}, \quad \lim\limits_{\zeta\to1}\frac{1-\zeta^p}{1-\zeta}=p,
$$
there is $C(m,p,\tau)>0$ such that
$$
\left(\frac{1-\zeta^p}{1-\zeta^m}\right)^{\tau}\left(\frac{1-\zeta^p}{1-\zeta}\right)^{2-\tau}\leq C(m,p,\tau), \quad \zeta\in[0,1),
$$
and the inequality~\eqref{ineq} follows then from~\eqref{interm1}, completing the proof.
\end{proof}

We are now in a position to give the proof of Theorem~\ref{th.uniq}, the main idea being a clever choice of weights that are singular at $x=0$ to be employed as test functions.

\begin{proof}[Proof of Theorem~\ref{th.uniq}]
Pick two initial conditions $u_{0,i}$, $i\in\{1,2\}$, as in Corollary~\ref{cor.exist} such that $u_{0,1}(x)\leq u_{0,2}(x)$ for any $x\in\mathbb{R}^N$. Let $u_i$, $i\in\{1,2\}$ be corresponding solutions to the Cauchy problem~\eqref{cp} on $[0,T)$ provided by Corollary~\ref{cor.exist} for some $T>0$ and introduce $U:=u_{1}-u_{2}$. Throughout the proof, we denote the positive part of $z\in\mathbb{R}$ by $z_+=\max\{z,0\}$ and set
\begin{equation*}
	\mathrm{sign}_+(x):=\left\{\begin{array}{ll}1, & {\rm if} \ x>0,\\ 0, & {\rm if} \ x\leq0.\end{array}\right.
\end{equation*}
Owing to~\eqref{wf2}, we may proceed as in the proof of \cite[Theorem, Eq.~(15)]{Ot96} to show that, for all non-negative $\vartheta\in C_c^1\big([0,T)\times\mathbb{R}^N\big)$,
\begin{align*}
	\int_0^T \int_{\mathbb{R}^N} \left[ U(0)_+ - U_+ \right] \partial_t\vartheta\,dxds & + \int_0^T \int_{\mathbb{R}^N} \mathrm{sign}_+\big( u_{1}^m - u_{2}^m \big) \nabla\big( u_{1}^m - u_{2}^m \big) \cdot\nabla \vartheta\,dxds \\
	& \quad - \int_0^T \int_{\mathbb{R}^N} |x|^\sigma \mathrm{sign}_+\big( u_{1} - u_{2} \big) \big( u_{1}^p - u_{2}^p \big) \vartheta\,dxds \le 0.
\end{align*}
Since $\mathrm{sign}_+\big( u_{1}^m - u_{2}^m \big) \nabla\big( u_{1}^m - u_{2}^m \big) = \nabla\big( u_{1}^m - u_{2}^m \big)_+$ a.e. in $(0,T)\times\mathbb{R}^N$ and $U(0)_+=0$ a.e. in $\mathbb{R}^N$, a further integration by parts gives
\begin{align*}
	- \int_0^T \int_{\mathbb{R}^N} U_+ \partial_t\vartheta\,dxds & - \int_0^T \int_{\mathbb{R}^N} \big( u_{1}^m - u_{2}^m \big)_+ \Delta \vartheta\,dxds \\
	& \qquad \le \int_0^T \int_{\mathbb{R}^N} |x|^\sigma \big( u_{1}^p - u_{2}^p \big)_+ \vartheta\,dxds.
\end{align*}
In view of the integrability and boundedness of $u_i(s)$ for $s\in (0,T)$ and $i\in\{1,2\}$, a classical truncation argument allows us to take $\vartheta=\xi w$ as a test function in the above inequality with $\xi\in C_c^1([0,T))$, $\xi\ge 0$, and $w\in C_+^{2}\big(\mathbb{R}^N\setminus\{0\}\big)$ satisfying
\begin{equation}
	w\in W^{2,1}(B(0,1)), \qquad w\in W^{2,\infty}\big(\mathbb{R}^N\setminus B(0,1)\big), \qquad x\mapsto |x|^\sigma w(x)\in L^1(B(0,1)), \label{sup02}
\end{equation}
and find
\begin{align*}
	- \int_0^T \int_{\mathbb{R}^N} U_+ w \partial_t\xi\,dxds & - \int_0^T \int_{\mathbb{R}^N} \big( u_{1}^m - u_{2}^m \big)_+ \xi \Delta w\,dxds \\
	& \qquad \le \int_0^T \int_{\mathbb{R}^N} |x|^\sigma \big( u_{1}^p - u_{2}^p \big)_+ \xi w\,dxds.
\end{align*}
We next fix $t\in (0,T)$ and choose appropriate non-negative approximations $\xi\in C_c^1([0,T))$ of the indicator function $\mathbf{1}_{[0,t]}$ of the interval $[0,t]$ in the above inequality to obtain
\begin{equation}
\begin{split}
	\int_{\mathbb{R}^N} U_+(t,x) w(x)\,dx & - \int_0^t \int_{\mathbb{R}^N} \big( u_{1}^m - u_{2}^m \big)_+(s,x) \Delta w(x)\,dxds \\
	& \qquad \le \int_0^t \int_{\mathbb{R}^N} |x|^\sigma w(x) \big( u_{1}^p - u_{2}^p \big)_+(s,x)\,dxds.
\end{split}  \label{uniq1}
\end{equation}
We are now at the point of choosing appropriately the weight $w$. This choice will be different in dimension $N\geq3$, $N=2$ and $N=1$, and we thus split the rest of the proof into three cases.

\medskip

\noindent $\triangleright\ N\geq3$. We fix $a\in(2-N,0)$ and choose $w(x) = |x|^{a}$, $x\in\mathbb{R}^N\setminus\{0\}$. Then $w\in C_+^{2}(\mathbb{R}^N\setminus\{0\})$ and, since
\begin{equation}
	\nabla w(x) = a |x|^{a-2} x, \quad \Delta w(x) = a (N+a-2) |x|^{a-2}< 0, \qquad x\in\mathbb{R}^N\setminus\{0\},  \label{sup01}
\end{equation}
the choice of $a$ and the constraints on $\sigma$ in~\eqref{exp} readily entail that $w$ satisfies~\eqref{sup02}.
We then infer from~\eqref{uniq1} and~\eqref{sup01} that
\begin{equation}
\begin{split}
	\int_{\mathbb{R}^N} |x|^a U_+(t,x)\,dx & + |a|(N+a-2) \int_0^t \int_{\mathbb{R}^N} |x|^{a-2} \big( u_{1}^m - u_{2}^m \big)_+(s,x)\,dxds \\
	& \qquad \le \int_0^t \int_{\mathbb{R}^N} |x|^{\sigma+a} \big( u_{1}^p - u_{2}^p \big)_+(s,x)\,dxds.
\end{split} \label{interm2}
\end{equation}

\smallskip

\noindent\textbf{Case~1.1: $N\ge 3$ and $p\ge p_G$.} Recalling that the constant $p_G$ is defined in~\eqref{pLexp}, we now assume that $p\geq p_G$ and infer from~\eqref{interm2}, Lemma~\ref{lem.ineq} (with $\tau=-\sigma\in(0,2)$) and Young's inequality that
\begin{align*}
\int_{\mathbb{R}^N}& |x|^a U_{+}(t,x)\,dx + |a|(N+a-2) \int_0^t \int_{\mathbb{R}^N} |x|^{a-2}(u_1^m-u_2^m)_{+}(s,x)\,dxds\\
&\leq C \int_0^t \int_{\mathbb{R}^N} |x|^{\sigma+a} \big[ \max\{u_1,u_2\}^{p-p_G} (u_1^m-u_2^m)_+^{-\sigma/2} U_{+}^{(2+\sigma)/2}\big](s,x)\,dxds\\
&\leq |a|(N+a-2) \int_0^t \int_{\mathbb{R}^N}|x|^{a-2}(u_1^m-u_2^m)_{+}(s,x)\,dxds\\
& \qquad + C \int_{\mathbb{R}^N} |x|^a \max\{u_1,u_2\}^{2(p-p_G)/(2+\sigma)}(s,x) U_{+}(s,x)\,dxds.
\end{align*}
Consequently, since $p\geq p_G$ and $u_i\in L^{\infty}((0,T)\times\mathbb{R}^N)$, $i\in\{1,2\}$, we readily conclude that
\begin{equation}\label{interm4}
\int_{\mathbb{R}^N} |x|^a U_{+}(t,x)\,dx \leq C(T) \int_0^t \int_{\mathbb{R}^N} |x|^{a} U_{+}(s,x)\,dxds
\end{equation}
for any $t\in[0,T]$, and an application of Gronwall's lemma gives
\begin{equation*}
	\int_{\mathbb{R}^N} |x|^a U_{+}(t,x)\,dx = 0,
\end{equation*}
which completes the proof in this case.

\smallskip

\noindent\textbf{Case~1.2: $N\ge 3$ and $p\in (1,p_G)$.} Assume now that $1<p<p_G$ and there is $\delta>0$ such that $u_{0,2}(x)\geq u_{0,1}(x)>0$ for $x\in B(0,\delta)$. Then, owing to the expansion of the support of non-negative solutions to the porous medium equation \cite[Proposition~9.19]{Va2007}, the solution $\overline{u}_1$ to
\begin{align*}
	\partial_t \overline{u}_1 - \Delta \overline{u}_1^m & = 0, \qquad t>0, \ x\in\real^N, \\
	\overline{u}_1(0) & = u_{0,1}, \qquad \ x\in\real^N,
\end{align*}
is positive on $[0,T]\times B(0,\delta)$, and we infer from the continuity of $\overline{u}_1$ and a comparison argument that there is $\nu>0$ such that
\begin{equation}\label{interm3}
\min\{u_1(s,x),u_2(s,x)\}\geq \overline{u}_1(s,x)\geq\nu, \quad (s,x)\in[0,T]\times B(0,\delta).
\end{equation}
We now handle differently the contributions of $B(0,\delta)$ and $\mathbb{R}^N\setminus B(0,\delta)$ in the right hand side of ~\eqref{interm2}, using again Lemma~\ref{lem.ineq} (with $\tau=-\sigma\in(0,2)$) and Young's inequality for the former, as follows:
\begin{align*}
\int_{\mathbb{R}^N}& |x|^a U_{+}(t,x)\,dx + |a|(N+a-2) \int_0^t \int_{\mathbb{R}^N} |x|^{a-2} (u_1^m-u_2^m)_{+}(s,x)\,dxds\\
&\leq C \int_0^t \int_{B(0,\delta)} |x|^{\sigma+a} \big[ \max\{u_1,u_2\}^{p-p_G} (u_1^m-u_2^m)_+^{-\sigma/2} U_{+}^{(2+\sigma)/2}\big](s,x)\,dxds\\
& \qquad + p\delta^{\sigma} \int_0^t \int_{\mathbb{R}^N\setminus B(0,\delta)} |x|^a (\|u_1(s)\|_{\infty}+\|u_2(s)\|_{\infty})^{p-1} U_{+}(s,x)\,dxds\\
& \le |a|(N+a-2) \int_0^t \int_{B(0,\delta)} |x|^{a-2} (u_1^m-u_2^m)_{+}(s,x)\,dxds \\
& \qquad + C \int_0^t \int_{B(0,\delta)} |x|^a \max\{u_1,u_2\}^{2(p-p_G)/(\sigma+2)}(s,x) U_{+}(s,x)\,dxds\\
& \qquad + p\delta^{\sigma} \int_0^t \int_{\mathbb{R}^N\setminus B(0,\delta)} |x|^a (\|u_1(s)\|_{\infty}+\|u_2(s)\|_{\infty})^{p-1} U_{+}(s,x)\,dxds.
\end{align*}
Since $p<p_G$, we infer from~\eqref{interm3} that
\begin{equation*}
	\max\{u_1,u_2\}^{2(p-p_G)/(\sigma+2)}(s,x)\leq\nu^{2(p-p_G)/(\sigma+2)}, \qquad (s,x)\in (0,T)\times B(0,\delta),
\end{equation*}
and use the boundedness of $u_1$ and $u_2$ to arrive to~\eqref{interm4} (with a different constant, which depends also on $\delta$) and thus conclude the proof.

\medskip

\noindent $\triangleright\ N=2$. We pick $a\in(0,1)$ such that
\begin{equation*}
	a^2 - (1+3\ln{2}) a + 2(\ln{2})^2 > 0
\end{equation*}
and set
\begin{equation*}
	w(x)=\left\{\begin{array}{ll}
	(-\ln\,|x|)^a, & x\in B(0,1/2), \\
	\displaystyle{b_0 + \frac{b_1}{|x|} - \frac{b_2}{|x|^2}}, & x\in\mathbb{R}^2\setminus B(0,1/2),
	\end{array}\right.
\end{equation*}
with 
\begin{align*}
	b_0 & := \frac{(\ln{2})^{a-2}}{2} \left[ a^2 -(1+3\ln{2}) a  + 2(\ln{2})^2 \right] >0, \\
	b_1 & := \frac{a}{2} (\ln{2})^{a-2} (1-a+2\ln{2})>0, \quad b_2 := \frac{a}{8} (\ln{2})^{a-2} (1-a+\ln{2})>0.
\end{align*}
Then $w\in C_+^2\big(\mathbb{R}^2\setminus\{0\}\big)$ with
\begin{equation*}
	\nabla w(x) = \left\{\begin{array}{ll}
	\displaystyle{- \frac{a(-\ln\,|x|)^{a-1}}{|x|^2} x}, & x\in B(0,1/2), \\
	\displaystyle{\left( - \frac{b_1}{|x|^3} + \frac{2b_2}{|x|^4} \right) x}, & x\in\mathbb{R}^2\setminus B(0,1/2),
	\end{array}\right.
\end{equation*}
and
\begin{equation*}
	\Delta w(x)=\left\{\begin{array}{ll}
	\displaystyle{-\frac{a(1-a)(-\ln\,|x|)^{a-2}}{|x|^2}}, & x\in B(0,1/2), \\
	\displaystyle{\frac{b_1}{|x|^3} - 4 \frac{b_2}{|x|^4}}, & x\in\mathbb{R}^2\setminus B(0,1/2),
	\end{array}\right.
\end{equation*}
 and satisfies~\eqref{sup02}. We then infer from~\eqref{uniq1}, Lemma~\ref{lem.ineq} (with $\tau\in(-\sigma,2)$) and Young's inequality that, for $\varepsilon\in (0,1/2]$,
\begin{align*}
	\int_{\mathbb{R}^2} & w(x) U_{+}(t,x)\,dx\\
	& \leq -a(1-a) \int_0^t \int_{B(0,\varepsilon)} \frac{(-\ln\,|x|)^{a-2}}{|x|^2} (u_1^m-u_2^m)_{+}(s,x)\,dxds\\
	& \quad + \int_0^t \int_{\mathbb{R}^2\setminus B(0,\varepsilon)}\ |\Delta w(x)| (u_1^m-u_2^m)_{+}(s,x)\,dxds\\
	& \quad + C \int_0^t \int_{B(0,\varepsilon)} |x|^{\sigma} (-\ln\,|x|)^a \big[\max\{u_1,u_2\}^{[2(p-1)-\tau(m-1)]/2}(u_1^m-u_2^m)_{+}^{\tau/2} U_{+}^{(2-\tau)/2}\big](s,x)\,dxds\\
	& \quad + \varepsilon^{\sigma} \int_0^t \int_{\mathbb{R}^2\setminus B(0,\varepsilon)} w(x) (u_1^p-u_2^p)_{+}(s,x)\,dxds\\
	& \leq -a(1-a) \int_0^t \int_{B(0,\varepsilon)} \frac{(-\ln\,|x|)^{a-2}}{|x|^2} (u_1^m-u_2^m)_{+}(s,x)\,dxds\\
	& \quad + a(1-a) \int_0^t \int_{B(0,\varepsilon)} \frac{(-\ln\,|x|)^{a-2}}{|x|^2} (u_1^m-u_2^m)_{+}(s,x)\,dxds\\
	& \quad + C \int_0^t \int_{B(0,\varepsilon)} |x|^{2(\sigma+\tau)/(2-\tau)} (-\ln\,|x|)^{2\tau/(2-\tau)} \max\{u_1,u_2\}^{[2(p-1)-\tau(m-1)]/(2-\tau)}(s,x) \\
	& \hspace{12cm} \times w(x) U_{+}(s,x)\,dxds\\
	& \quad + C(\varepsilon) \int_0^t \int_{\mathbb{R}^2\setminus B(0,\varepsilon)} \left[ (\|u_1(s)\|_{\infty} + \|u_2(s)\|_{\infty})^{m-1} + (\|u_1(s)\|_{\infty} + \|u_2(s)\|_{\infty})^{p-1} \right] \\
	& \hspace{12cm} \times w(x) U_{+}(s,x)\,dxds\\
	&  \leq C \int_0^t \int_{B(0,\varepsilon)} |x|^{2(\sigma+\tau)/(2-\tau)} (-\ln\,|x|)^{2\tau/(2-\tau)} \max\{u_1,u_2\}^{[2(p-1)-\tau(m-1)]/(2-\tau)}(s,x)\\
	& \hspace{12cm} \times w(x) U_{+}(s,x)\,dxds\\
	& \quad + C(T,\varepsilon) \int_0^t \int_{\mathbb{R}^2} w(x) U_{+}(s,x)\,dxds,
\end{align*}
the last inequality being deduced from the boundedness of $u_i$, $i\in\{1,2\}$, and~\eqref{exp}.
Owing to the choice $\tau>-\sigma$, the function
\begin{equation*}
	x\mapsto|x|^{2(\sigma+\tau)/(2-\tau)}(-\ln\,|x|)^{2\tau/(2-\tau)}
\end{equation*}
is bounded on $B(0,\varepsilon)$, and we end up with
\begin{align*}
	\int_{\mathbb{R}^2} w(x) U_{+}(t,x)\,dx & \leq C \int_0^t \int_{B(0,\varepsilon)} \max\{u_1,u_2\}^{[2(p-1)-\tau(m-1)]/(2-\tau)}(s,x) w(x) U_{+}(s,x)\,dxds\\
		& + C(T,\varepsilon) \int_0^t \int_{\mathbb{R}^2} w(x) U_{+}(s,x)\,dxds.
\end{align*}

\smallskip

\noindent\textbf{Case~2.1: $N=2$ and $p>p_G$.} We take $\varepsilon=1/2$ and observe that, in that case, there is $\tau\in(-\sigma,2)$ such that $2(p-1)-\tau(m-1)>0$. We then proceed as in \textbf{Case~1.1} to conclude.

\smallskip

\noindent\textbf{Case~2.2: $N=2$ and $p\in (1,p_G]$.} Then $2(p-1)-\tau(m-1)<0$ for any $\tau\in(-\sigma,2)$ and we proceed as in \textbf{Case~1.2} with $\varepsilon=\delta$ and the help of~\eqref{interm3}, which is valid due to the assumed positivity of both $u_{0,1}$ and $u_{0,2}$.

\medskip

\noindent $\triangleright\ N=1$. Let
\begin{equation}
	a \in \left( \frac{1}{2} , \frac{\sqrt{17}-3}{2} \right) \cap \left( \frac{1}{2} , 1 + \frac{\sigma}{2} \right) \label{sup03}
\end{equation}
to be determined later and notice that the lower bound $\sigma>-1$ entails that such a choice is always possible. We next define
\begin{equation*}
w(x) = \left\{\begin{array}{ll}
	2 - |x|^{2a}, & x\in (-1,1), \\
	\displaystyle{b_0 + \frac{b_2}{x^2} - \frac{b_4}{x^4}}, & x\in\mathbb{R}\setminus (-1,1),
\end{array}\right.
\end{equation*}
with
\begin{equation*}
	b_0 := \frac{2-3a-a^2}{2}>0, \quad b_2 := a(a+2)>0, \quad b_4 := \frac{a(a+1)}{2}>0,
\end{equation*}
the positivity of $b_0$ being guaranteed by\eqref{sup03}. Then $w\in C_+^2(\mathbb{R}\setminus\{0\}\big)$ with
\begin{equation*}
	w'(x) = \left\{\begin{array}{ll}
		- 2a |x|^{2a-2} x, & x\in (-1,1), \\
		\displaystyle{- \frac{2b_2}{x^3} + \frac{4b_4}{x^5}}, & x\in\mathbb{R}\setminus (-1,1),
	\end{array}\right.
\end{equation*}
and
\begin{equation*}
w''(x) = \left\{\begin{array}{ll}
	- 2a (2a-1) |x|^{2a-2}, & x\in (-1,1), \\
		\displaystyle{\frac{6b_2}{x^4} - \frac{20 b_4}{x^6}}, & x\in\mathbb{R}\setminus (-1,1),
	\end{array}\right.
\end{equation*}
In addition, $w$ satisfies~\eqref{sup02} according to~\eqref{sup03}. We then proceed as in the two dimensional case to obtain that, for $\varepsilon\in (0,1]$ and $\tau\in (0,2)$,
\begin{align*}
\int_{\mathbb{R}}& w(x) U_{+}(t,x)\,dx \\
& \leq C(\varepsilon) \int_0^t \int_{-\varepsilon}^{\varepsilon}  \max\{u_1,u_2\}^{[2(p-1)-\tau(m-1)]/(2-\tau)}(s,x) |x|^{2[\sigma+\tau(1-a)]/(2-\tau)} w(x) U_{+}(s,x)\,dxds\\
& \quad + C(T,\varepsilon) \int_0^t \int_{\mathbb{R}} w(x) U_{+}(s,x)\,dxds.
\end{align*}
At this point, we notice that $-\sigma/(1-a)\in (-2\sigma,2)\subset (0,2)$ by~\eqref{sup03} and we choose $\tau=-\sigma/(1-a)$ in the above inequality to conclude that
\begin{align*}
	\int_{\mathbb{R}} w(x) U_{+}(t,x)\,dx & \leq C(\varepsilon) \int_0^t \int_{-\varepsilon}^{\varepsilon}  \max\{u_1,u_2\}^{[2(p-1)-\tau(m-1)]/(2-\tau)}(s,x) w(x) U_{+}(s,x)\,dxds\\
	& \quad + C(T,\varepsilon) \int_0^t \int_{\mathbb{R}} w(x) U_{+}(s,x)\,dxds.
\end{align*}

\smallskip

\noindent\textbf{Case~3.1: $N=1$ and $p>1-\sigma(m-1)$.} We take $\varepsilon=1$ and choose $a>1/2$ sufficiently close to $1/2$ in the range~\eqref{sup03} such that
\begin{equation*}
	2 (p-1) - \tau (m-1) = 2(p-1+\sigma(m-1)) + \frac{\sigma(m-1)}{1-a} (2a-1)\ge 0,
\end{equation*}
and we proceed as in \textbf{Case~2.1} to complete the proof.

\smallskip

\noindent\textbf{Case~3.2: $N=1$ and $p\in (1,1-\sigma(m-1)]$.} In that case,
\begin{equation*}
	2 (p-1) - \tau (m-1) = 2(p-1+\sigma(m-1)) + \frac{\sigma(m-1)}{1-a} (2a-1)\le 0,
\end{equation*}
and we proceed as in \textbf{Case~2.2} to complete the proof, with $\varepsilon=\delta$ provided by~\eqref{interm3}.
\end{proof}

\subsection{Maximal solution}\label{sec.ms}

We now turn to the existence of a unique weak solution to~\eqref{cp} in the sense of Corollary~\ref{cor.exist} defined on a maximal interval time, as stated in Corollary~\ref{cor.maxwp} when $p>p_G$. The starting point is a consequence of the comparison principle for the approximate problem~\eqref{cp.eta}, along with an extension of the proof of Theorem~\ref{th.exist}.

\begin{lemma}\label{lem.x}
	Consider $p>p_G$ and let $u_0\in L_+^1(\mathbb{R}^N)\cap L^\infty(\mathbb{R}^N)$ and $T>0$. If $u$ is a weak solution to~\eqref{cp} on $[0,T)$ in the sense of Corollary~\ref{cor.exist}, then
	\begin{equation}
		u(t,x) \ge u_\eta(t,x), \qquad (t,x)\in [0,T)\times\mathbb{R}^N, \label{zz01}
	\end{equation}
	for all $\eta\in (0,1)$, recalling that $u_\eta$ is the solution to~\eqref{cp.eta}. Moreover,
	\begin{equation}
		u(t,x) = \bar{u}(t,x) := \sup_{\eta\in (0,1)} u_\eta(t,x) , \qquad (t,x)\in [0,T)\times\mathbb{R}^N. \label{zz02}
	\end{equation}
\end{lemma}

\begin{proof}
Let $\eta\in (0,1)$. Since
\begin{equation*}
	\left( |x|^2 + \eta^2 \right)^{\sigma/2} \frac{u^p(t,x)}{1+\eta u^p(t,x)} \le |x|^\sigma u^p(t,x), \qquad (t,x)\in [0,T)\times\mathbb{R}^N,
\end{equation*}
and $u_0\ge u_{0,\eta}$ in $\mathbb{R}^N$, the function $u$ is a supersolution to~\eqref{cp.eta} and the ordering~\eqref{zz01} is an immediate consequence of the comparison principle applied to~\eqref{eq.eta}. It next follows from the monotonicity~\eqref{ex17} of $(u_\eta)_{\eta\in (0,1)}$ with respect to $\eta$ and~\eqref{zz01} that $\bar{u}$ is well-defined in $[0,T)\times\mathbb{R}^N$ and belongs to $L_{\mathrm{loc}}^\infty\big([0,T),L^1(\mathbb{R}^N)\cap L^\infty(\mathbb{R}^N)\big)$. We then argue as in the proof of Theorem~\ref{th.exist} to conclude that $\bar{u}$ is a weak solution to~\eqref{cp} on $[0,T)$ in the sense of Corollary~\ref{cor.exist}, which implies that $\bar{u}=u$ by Theorem~\ref{th.uniq} since $p>p_G$.
\end{proof}

\begin{proof}[Proof of Corollary~\ref{cor.maxwp}]
Introducing
\begin{equation*}
	T_{\mathrm{max}}(u_0) := \sup\left\{ \tau>0\ :\ \bar{u} = \sup_{\eta\in (0,1)} u_\eta \in L^\infty((0,\tau)\times\mathbb{R}^N) \right\}
\end{equation*}
we first infer from the proof of Theorem~\ref{th.exist} that $T_{\mathrm{max}}(u_0)\ge T_\infty>0$.

We next argue as in the last step of the proof of Theorem~\ref{th.exist} (with $q=1$ and $q=m+1$) to establish that
\begin{equation*}
	\bar{u} \in L^\infty\big((0,\tau),L^1(\mathbb{R}^N)\big) \;\;\text{ with }\;\; \nabla\bar{u}^m \in L^2\big((0,\tau)\times\mathbb{R}^N\big)
\end{equation*}
for all $\tau\in (0,T_{\mathrm{max}}(u_0))$. Proceeding along the lines of the proof of Theorem~\ref{th.exist}, we then show that $\bar{u}$ is a weak solution to~\eqref{cp} on $[0,T_{\mathrm{max}}(u_0))$ in the sense of Corollary~\ref{cor.exist}.

Assume now for contradiction that there is a weak solution $u$ to~\eqref{cp} on $[0,T)$ in the sense of Corollary~\ref{cor.exist} for some $T>T_{\mathrm{max}}(u_0)$. Then $\bar{u}=u$ belongs to $L^\infty\big((0,\tau)\times\mathbb{R}^N\big)$ for any $\tau\in (T_{\mathrm{max}}(u_0),T)$ by Lemma~\ref{lem.x}, which contradicts the maximality of $T_{\mathrm{max}}(u_0)$.

Finally, assume for contradiction that $T_{\mathrm{max}}(u_0)<\infty$ with
\begin{equation*}
	\sup_{t\in [0,T_{\mathrm{max}}(u_0)]} \|\bar{u}(t)\|_\infty < \infty.
\end{equation*}
We then argue again as in the last step of the proof of Theorem~\ref{th.exist} to deduce that
\begin{equation*}
	\sup_{t\in [0,T_{\mathrm{max}}(u_0)]} \|\bar{u}(t)\|_1 < \infty.
\end{equation*}
In particular, $v_0 := \bar{u}(T_{\mathrm{max}}(u_0))\in L_+^1(\mathbb{R}^N)\cap L^\infty(\mathbb{R}^N)$ and it follows from Corollary~\ref{cor.exist} and Theorem~\ref{th.uniq} that the Cauchy problem~\eqref{cp} with initial condition $v_0$ has a unique weak solution $v$ defined on some non-empty time interval $[0,T_0)$. Then the function $\bar{v}$ defined by
\begin{equation*}
	\bar{v}(t,x) := \left\{ \begin{array}{ll}
	\bar{u}(t,x), & (t,x)\in [0,T_{\mathrm{max}}(u_0)]\times\mathbb{R}^N, \\
	v(t-T_{\mathrm{max}}(u_0),x), & t\in [T_{\mathrm{max}}(u_0),T_{\mathrm{max}}(u_0)+T_0),
	\end{array}\right.
\end{equation*}
is a weak solution to~\eqref{cp} on $[0,T_{\mathrm{max}}(u_0)+T_0)$ in the sense of Corollary~\ref{cor.exist}, a property which contradicts the definition of $T_{\mathrm{max}}(u_0)$ according to Lemma~\ref{lem.x}. Therefore, if $T_{\mathrm{max}}(u_0)<\infty$, then
\begin{equation*}
	\limsup_{t\nearrow T_{\mathrm{max}}(u_0)} \|\bar{u}(t)\|_\infty=\infty,
\end{equation*}
 and the proof is complete.
\end{proof}

\section{Finite time blow-up}\label{sec.bu}

The goal of this section is to prove Theorems~\ref{th.blowup} and~\ref{th.blowup2}. The former is established by employing different approaches according to the range of $p$, as explained in Section~\ref{sec.mr}. Thus, its proof is split into three sections covering the ranges $p\in(p_G,m)$, $p=m$ and $p\in(m,p_F]$, respectively. The subsequent Theorem~\ref{th.blowup2} is then proved in Section~\ref{sec.bu4}.

Throughout this section, we assume that $p>p_G$ and consider $u_0\in L_+^1(\mathbb{R}^N)\cap L^\infty(\mathbb{R}^N)$, $u_0\not\equiv 0$. According to Corollary~\ref{cor.maxwp}, the Cauchy problem~\eqref{cp} has a unique non-negative weak solution $u$ defined on a \textsl{maximal} time interval $[0,T_{\mathrm{max}}(u_0))$ and the aim of this section is to show the finiteness of $T_{\mathrm{max}}(u_0)$ in various situations.

\subsection{Finite time blow-up: $p\in (p_G,m)$} \label{sec.bu1}
We begin with exponents $p\in(p_G,m)$, noticing that this is a non-empty interval since $\sigma>-2$.
\begin{proposition}\label{prop.bu1}
Assume that $p\in(p_G,m)$. Then $T_{\mathrm{max}}(u_0)<\infty$.
\end{proposition}
The proof relies on the comparison principle for Eq~\eqref{eq1} established in Theorem~\ref{th.uniq} and the construction of a suitable subsolution, the latter being a straightforward adaptation of \cite[Chapter~IV, \S~3.1]{SGKM1995} which we sketch for the sake of completeness.

\begin{lemma}\label{lem.bu10}
	For $(T,A,a)\in (0,\infty)^3$, we define
	\begin{equation*}
		S_{T,A,a}(t,x) := (T-t)^{-\alpha} s\big( |x| (T-t)^\beta \big), \qquad (t,x)\in [0,T)\times\real^N,
	\end{equation*}
	with
	\begin{equation*}
		s(y) := A \left( 1 - \frac{y^2}{a^2} \right)_+^{1/(m-1)}, \qquad y\in \real,
	\end{equation*}
	and
	\begin{equation*}
		\alpha := \frac{\sigma+2}{2(p-p_G)}>0, \quad \beta := \frac{m-p}{2(p-p_G)}>0.
	\end{equation*}
There is $A_0>0$ depending only on $N$, $m$, $p$ and $\sigma$ such that, for $A\ge A_0$ and $a^2=mA^{m-1}/[\beta(m-1)]$,
	\begin{equation}
		\partial_t S_{T,A,a}(t,x) - \Delta S_{T,A,a}^m(t,x) - |x|^\sigma S_{T,A,a}^p(t,x) \le 0, \qquad (t,x)\in [0,T)\times\real^N.
	\end{equation}
\end{lemma}

\begin{proof}
	Introducing $R(t,x) := \partial_t S_{T,A,a}(t,x) - \Delta S_{T,A,a}^m(t,x) - |x|^\sigma S_{T,A,a}^p(t,x)$ and $y=|x| (T-t)^\beta$ for $(t,x)\in [0,T)\times\real^N$, we find
	\begin{align*}
		R(t,x) & = (T-t)^{-\alpha-1} \left[ \alpha s(y) - \beta y s'(y) \right] - (T-t)^{2\beta-\alpha m} \left[ \left( s^m \right)''(y) + \frac{N-1}{y} \left( s^m \right)'(y) \right] \\
		& \qquad - (T-t)^{-\alpha p - \beta\sigma} y^\sigma s^p(y).
	\end{align*}
	Owing to the choice of $\alpha$ and $\beta$,
	\begin{equation*}
		\alpha+1  = \alpha m - 2 \beta = \alpha p + \sigma \beta = \frac{2p+m\sigma}{2(p-p_G)},
	\end{equation*}
	so that, setting also $z:= 1- y^2/a^2$,
	\begin{align*}
		(T-t)^{(2p+m\sigma)/[2(p-p_G)]} R(t,x) & = \left[ \alpha A + \frac{2mN A^m}{(m-1) a^2} \right] z_+^{1/(m-1)} \\
		& \qquad - \left[ \frac{4m A^m}{(m-1)^2a^2} - \frac{2\beta A}{m-1} \right] (1-z) z_+^{(2-m)/(m-1)} - A^p y^\sigma z_+^{p/(m-1)} \\
		& = \left[ \frac{A}{m-1} + \frac{2m[N(m-1)+2] A^m}{(m-1)^2 a^2} \right] z_+^{1/(m-1)} \\
		& \qquad - \frac{2A}{m-1} \left( \frac{2m A^{m-1}}{(m-1)a^2} - \beta \right) z_+^{(2-m)/(m-1)} - A^p y^\sigma z_+^{p/(m-1)}.
	\end{align*}
	Hence, for $(t,x)\in [0,T)\times\real^N$ such that $z\in (0,1]$,
	\begin{align*}
		R_1(t,x) & := (m-1)^2 z^{(m-2)/(m-1)} (T-t)^{(2p+m\sigma)/[2(p-p_G)]} R(t,x) \nonumber \\
		& =\left[ A(m-1) + \frac{2m[N(m-1)+2] A^m}{a^2} \right] z - 2A \left( \frac{2m A^{m-1}}{a^2} - \beta(m-1) \right) \\
		& \hspace{4cm} - (m-1)^2A^p y^\sigma z^{(p+m-2)/(m-1)}.
	\end{align*}
	First, in order to guarantee the non-positivity of the second term on the right hand side of the above identity, we assume that
	\begin{equation}
		A^{m-1} \ge \frac{\beta (m-1)a^2}{m}, \label{bu12}
	\end{equation}
	from which we deduce that
	\begin{equation}
		R_1(t,x) \le \left[ A(m-1) + \frac{2m[N(m-1)+2]A^m}{a^2} \right] z - \frac{2m A^{m}}{a^2} - (m-1)^2 A^p y^\sigma z^{(p+m-2)/(m-1)}. \label{bu13}
	\end{equation}
	Next, pick $z_0\in (0,1)$ to be determined later and set $y_0:= a \sqrt{1-z_0} \in(0,a)$. Either $z\in [0,z_0]$ and, if
	\begin{equation}
		z_0 \le \min\left\{ \frac{mA^{m-1}}{(m-1)a^2} , \frac{1}{2[N(m-1)+2]} \right\}, \label{bu14}
	\end{equation}
	then we infer from~\eqref{bu13} that
	\begin{equation*}
		R_1(t,x) \le A(m-1) z_0 -  \frac{m A^{m}}{a^2} + \frac{2m[N(m-1)+2] A^m}{a^2} z_0 - \frac{m A^{m}}{a^2} \le 0.
	\end{equation*}
	Or $z>z_0$. Then $y<y_0$ and it follows from~\eqref{bu13} and the negativity of $\sigma$ that
	\begin{equation}
		R_1(t,x) \le A \left[m-1+ \frac{2m[N(m-1)+2] A^{m-1}}{a^2} - (m-1)^2 A^{p-1} y_0^\sigma z_0^{(p-1)/(m-1)} \right] z \le 0, \label{bu15}
	\end{equation}
	provided
	\begin{equation}
		a^\sigma A^{p-1} \ge \frac{(1-z_0)^{-\sigma/2}}{(m-1)^2 z_0^{(p-1)/(m-1)}} \left[m-1+ \frac{2m[N(m-1)+2]A^{m-1}}{a^2} \right]. \label{bu16}
	\end{equation}
	At this point, we choose
	\begin{equation}
		a = A^{(m-1)/2} \sqrt{\frac{m}{\beta(m-1)}} \;\;\text{ and }\;\; z_0 := \min\left\{\frac{1}{2[N(m-1)+2]} , \beta \right\}\in (0,1), \label{bu17}
	\end{equation}
	so that~\eqref{bu12} and~\eqref{bu14} are satisfied, while~\eqref{bu16} reads
	\begin{equation*}
		\left( \frac{m}{\beta(m-1)} \right)^{\sigma/2} A^{p-p_G} \ge \frac{(1-z_0)^{-\sigma/2}}{(m-1) z_0^{(p-1)/(m-1)}} \left[ 1 + 2[N(m-1)+2]\beta \right],
	\end{equation*}
	or, equivalently,
	\begin{equation*}
		A^{p-p_G} \ge A_0^{p-p_G} := \left( \frac{m}{\beta(m-1)} \right)^{-\sigma/2} \frac{(1-z_0)^{-\sigma/2}}{(m-1) z_0^{(p-1)/(m-1)}} \left[ 1 + 2[N(m-1)+2]\beta \right].
	\end{equation*}
	The proof of Lemma~\ref{lem.bu10} is then complete.
\end{proof}

\begin{proof}[Proof of Proposition~\ref{prop.bu1}]
	Since $u_0\not\equiv 0$ and $u$ is a supersolution to the porous medium equation
	\begin{subequations}\label{pme}
	\begin{align}
		\partial_t U & = \Delta U^m, \qquad (t,x)\in (0,\infty)\times\real^N, \label{eq.pme} \\
		U(0) & = u_0, \qquad x\in\real^N, \label{ic.pme}
	\end{align}
	\end{subequations}
	classical properties of the porous medium equation imply that there are $\tau_0>0$ and $\eta_0>0$ such that
	\begin{equation}
		u(\tau_0,x) \ge U(\tau_0,x) \ge \eta_0, \qquad x\in B(0,a_0), \label{bu18}
	\end{equation}
	see \cite[Theorem~14.3]{Va2007} for instance, where $a_0^2 := m A_0^{m-1}/[\beta(m-1)]$ and $A_0$ is defined in Lemma~\ref{lem.bu10}. Setting $T_0:= (A_0/\eta_0)^{1/\alpha}$, we readily infer from~\eqref{bu18} and the definition of $S_{T_0,A_0,a_0}$ that
	\begin{equation*}
		u(\tau_0,x) \ge A_0 \frac{\eta_0}{A_0} \ge S_{T_0,A_0,a_0}(0,x), \qquad x\in\mathbb{R}^N,
	\end{equation*}
	which allows us to apply the comparison principle to~\eqref{eq1} in view of Theorem~\ref{th.uniq} and Lemma~\ref{lem.bu10} and conclude that
	$u(t+\tau_0,x)\ge S_{T_0,A_0,a_0}(t,x)$ for $(t,x)\in [0,T_0)\times\mathbb{R}^N$. We have thus shown that $u$ blows up in finite time, as stated.
\end{proof}

\subsection{Finite time blow-up: $p=m$} \label{sec.bu2}
In the case $p=m$, we give a different proof, based on previous results by the authors published in \cite{IL}, of the fact that any non-trivial solution blows up in finite time.
\begin{proposition}\label{prop.bu2}
Assume that $p=m$. Then $T_{\mathrm{max}}(u_0)<\infty$.
\end{proposition}

\begin{proof}
We adapt Kaplan's technique \cite{Ka1963} and first recall that, according to \cite{IL}, there is a non-negative compactly supported function
\begin{equation*}
	v_*\in \bigcap_{q\in [1,N/|\sigma|)} W^{2,q}(\real^N),
\end{equation*}
which solves
\begin{equation}
	- \Delta v_*(x) + \frac{1}{m-1} v_*^{1/m}(x) = |x|^\sigma v_*(x), \qquad x\in\mathbb{R}^N. \label{bu10}
\end{equation}
We infer from~\eqref{eq1}, \eqref{bu10}, and the properties of $v_*$ that
\begin{align*}
	\frac{d}{dt} \int_{\real^N} u v_*\ dx & = \int_{\real^N} u^m \big( \Delta v_* + |x|^\sigma v_* \big)\ dx = \frac{1}{m-1} \int_{\real^N} u^m v_*^{1/m}\ dx \\
	& \ge \frac{1}{(m-1)\|v_*\|_\infty^{(m-1)/m}} \int_{\real^N} u^m v_*\ dx \\
	& = \frac{\|v_*\|_1}{(m-1)\|v_*\|_\infty^{(m-1)/m}} \int_{\real^N} u^m \frac{v_*}{\|v_*\|_1}\ dx.
\end{align*}
Since $m>1$, it follows from Jensen's inequality that
\begin{align*}
	\frac{d}{dt} \int_{\real^N} u v_*\ dx & \ge \frac{\|v_*\|_1}{(m-1)\|v_*\|_\infty^{(m-1)/m}} \left( \int_{\real^N} u \frac{v_*}{\|v_*\|_1}\ dx \right)^m \\
	& = \frac{C(v_*)}{m-1} \left( \int_{\real^N} u v_*\ dx \right)^m.
\end{align*}
After integration with respect to time, we further obtain
\begin{equation*}
	\|u(t)\|_\infty \|v_*\|_1 \ge \int_{\real^N} u(t) v_*\ dx\ge \left[ \left( \int_{\real^N} u_0 v_*\ dx \right)^{1-m} - C(v_*)t \right]_+^{-1/(m-1)},
\end{equation*}
which readily ensures that $u$ blows up in finite time.
\end{proof}

\subsection{Finite time blow-up: $p\in \big(m,p_F\big]$} \label{sec.bu3}

As previously mentioned, Theorem~\ref{th.blowup} is already proved in \cite[Theorem~1.6]{Qi98} when $p$ ranges in $(m,p_F]$ and we do not provide a proof here, even though a slightly simpler proof can be performed by a scaling argument as in \cite[Theorem~26.1]{QuSo2019}, proceeding along the lines of \cite[Theorem~18.1(i)]{MiPo2001}.

\subsection{Finite time blow-up: $p>p_F$}\label{sec.bu4}

This section is dedicated to the proof of Theorem~\ref{th.blowup2}. Let us thus pick $m$, $p$ and $\sigma$ as in~\eqref{exp} with $p\geq m$, noticing that the novelty appears when $p>p_F$, but the proof remains true under this more general condition.
\begin{proof}[Proof of Theorem~\ref{th.blowup2}]
	We proceed along the lines of \cite{BeHo1980}, though Theorem~\ref{th.blowup2} may be also be proved by the concavity method introduced in \cite{Le1973}. For $t\in [0,T_{\mathrm{max}}(u_0))$, we set
	\begin{equation*}
		\mathcal{I}(t) := \|u(t)\|_{m+1}^{m+1}, \quad \mathcal{E}(t) := E(u(t)), \quad \mathcal{D}(t) := \int_{\mathbb{R}^N} \partial_t u \partial_t u^m\ dx \ge 0,
	\end{equation*}
recalling the energy $E(u(t))$ defined in~\eqref{energy}. On the one hand, we infer from~\eqref{eq1} that
	\begin{equation}
		\mathcal{E}'(t) \le - \mathcal{D}(t), \qquad t\in [0,T_{\mathrm{max}}(u_0)), \label{bu8}
	\end{equation}
	and
	\begin{align*}
		\mathcal{I}'(t) & = (m+1) \left[ - \|\nabla u^m(t)\|_2^2 + \int_{\mathbb{R}^N} |x|^\sigma u(t,x)^{m+p}\ dx \right] \\
		& = (m+1) \left[ - \frac{m+p}{m} \mathcal{E}(t) + \left( \frac{m+p}{2m} - 1 \right) \|\nabla u^m(t)\|_2^2 \right],
	\end{align*}
	whence, since $p\ge m$,
	\begin{equation}
		\mathcal{I}'(t) \ge -\frac{(m+1)(m+p)}{m} \mathcal{E}(t), \qquad t\in [0,T_{\mathrm{max}}(u_0)). \label{bu9}
	\end{equation}
	Moreover, an immediate consequence of~\eqref{bu6}, \eqref{bu8}, and the non-negativity of $\mathcal{D}$ is
	\begin{equation}
		\mathcal{E}(t) \le \mathcal{E}(0) = E(u_0) < 0, \qquad t\in [0,T_{\mathrm{max}}(u_0)). \label{bu20}
	\end{equation}
	On the other hand, we infer from H\"older's inequality that, for any $t\in [0,T_{\mathrm{max}}(u_0))$, we have
	\begin{equation}
\begin{split}
		\mathcal{I}'(t) &= (m+1) \int_{\mathbb{R}^N} u^m(t,x) \partial_t u(t,x)\ dx =\frac{m+1}{\sqrt{m}}\int_{\mathbb{R}^N}u^{(m+1)/2}(t,x)
(\partial_t u\partial_t u^m)^{1/2}(t,x)\,dx\\
&\le (m+1) \sqrt{\frac{\mathcal{I}(t)\mathcal{D}(t)}{m}}, \label{bu21}
\end{split}
	\end{equation}
	Combining~\eqref{bu9}, \eqref{bu20}, and~\eqref{bu21} gives
	\begin{equation*}
		0 < - \frac{m+p}{m} \mathcal{E} \le \frac{\mathcal{I}'}{m+1} \le \sqrt{\frac{\mathcal{I}\mathcal{D}}{m}},
	\end{equation*}
	from which we deduce, together with~\eqref{bu8}, \eqref{bu20}, and~\eqref{bu21}, that
	\begin{align*}
		\left( \frac{\mathcal{E}}{\mathcal{I}^{a}} \right)' & = \frac{\mathcal{E}'}{\mathcal{I}^a} - a \frac{\mathcal{E} \mathcal{I}'}{\mathcal{I}^{a+1}} \le - \frac{\mathcal{D}}{\mathcal{I}^a} - a (m+1) \frac{\mathcal{E}}{\mathcal{I}^{a+1}} \sqrt{\frac{\mathcal{I}\mathcal{D}}{m}} \\
		& = \sqrt{\frac{m\mathcal{D}}{\mathcal{I}^{2a+1}}} \left( - \sqrt{\frac{\mathcal{I}\mathcal{D}}{m}} - a \frac{m+1}{m} \mathcal{E} \right) \\
		& \le \sqrt{\frac{m\mathcal{D}}{\mathcal{I}^{2a+1}}} \left( \frac{m+p}{m} \mathcal{E} - a \frac{m+1}{m} \mathcal{E} \right) = 0,
	\end{align*}
	with $a:= (m+p)/(m+1)$. Consequently,
	\begin{equation}
		\mathcal{E}(t) \le \frac{\mathcal{E}(0)}{\mathcal{I}^a(0)} \mathcal{I}^a(t), \qquad t\in (0,T_{\mathrm{max}}(u_0)). \label{bu22}
	\end{equation}
	We then infer from~\eqref{bu9} and~\eqref{bu22} that
	\begin{equation*}
		\mathcal{I}'(t) \ge - \frac{(m+1)(m+p)}{m} \frac{E(u_0)}{\|u_0\|_{m+1}^{m+p}} \mathcal{I}^a(t), \qquad t\in (0,T_{\mathrm{max}}(u_0)).
	\end{equation*}
	On the one hand, the negativity~\eqref{bu6} of $E(u_0)$ and the property $a>1$ ensure that $\mathcal{I}$ blows up in finite time and thus $T_{\mathrm{max}}(u_0)<\infty$. On the other hand, integrating the above differential inequality over $(t,T_{\mathrm{max}}(u_0))$ gives
	\begin{equation*}
		0 - \frac{\mathcal{I}^{1-a}(t)}{1-a} \ge - \frac{(m+1)(m+p)}{m} \frac{E(u_0)}{\|u_0\|_{m+1}^{m+p}} \left[ T_{\mathrm{max}}(u_0) - t \right],
	\end{equation*}
	or, equivalently,
	\begin{equation*}
		\mathcal{I}^{1-a}(t) \ge \frac{(p-1)(m+p)}{m} \frac{|E(u_0)|}{\|u_0\|_{m+1}^{m+p}} \left[ T_{\mathrm{max}}(u_0) - t \right],
	\end{equation*}
which readily leads to the upper bound~\eqref{bu7}, as claimed.
\end{proof}

\section{Global solutions: $p>p_F$}\label{sec.gs}

The purpose of this section is to prove Theorem~\ref{th.global}. Specifically, we assume throughout this section that $(m,p,\sigma)$ satisfies~\eqref{exp} with $p>p_F$ and the initial condition $u_0$ is non-negative and belongs to the space $L^{r_0+1}(\mathbb{R}^N)\cap L_+^\infty(\mathbb{R}^N)$, with $r_0$ defined in~\eqref{rexp}. Notice that $r_0=r_c>0$ due to $p>p_F$. The starting point of the forthcoming estimates is the following result, which is a particular case of Proposition~\ref{prop.ex1}.

\begin{proposition}\label{prop.gs2}
Let $p>p_F$, $\bar{r}\in (r_0,\infty)$ and $r\in [r_0,\bar{r}]$. Consider $w\in L_+^{r_0+1}(\mathbb{R}^N)$ such that $w^{(m+r)/2}\in \dot{H}^1(\mathbb{R}^N)$. Then $w\in L^{p+r}(\mathbb{R}^N, |x|^\sigma dx)$ and
	\begin{equation}
		\int_{\mathbb{R}^N} |x|^\sigma w^{p+r}(x)\ dx \le \Lambda(\bar{r}) \left\| \nabla w^{(m+r)/2} \right\|_2^2 \left( \|w\|_{r_0+1}^{r_0+1} \right)^{(\sigma+2)/N}, \label{gs02}
	\end{equation}
the constant $\Lambda(\bar{r})\ge 1$ being defined in Proposition~\ref{prop.ex1}.
\end{proposition}

\begin{proof}
This result is a particular case of Proposition~\ref{prop.ex1} by letting $r_1=r_0$. Indeed, recalling the definitions of $\omega_r$ and $\mu_r$ in~\eqref{ex01}, we deduce that $\omega_r=1$ and $\mu_r=(\sigma+2)/N$ and the conclusion follows.
\end{proof}

\begin{proof}[Proof of Theorem~\ref{th.global}]
	For $r\in [r_0,r_0+1]$, it follows from~\eqref{eq1} and Proposition~\ref{prop.gs2} that
	\begin{align*}
		\frac{1}{r+1} \frac{d}{dt} \|u\|_{r+1}^{r+1} & = - \frac{4rm}{(m+r)^2} \left\| \nabla u^{(m+r)/2} \right\|_2^2 + \int_{\mathbb{R}^N} |x|^\sigma u^{p+r}\ dx \\
		& \le - \frac{4mr_0}{(m+r)(m+r_0)} \left\| \nabla u^{(m+r)/2} \right\|_2^2 +\Lambda(r_0+1)\left\| \nabla u^{(m+r)/2} \right\|_2^2 \left( \|u\|_{r_0+1}^{r_0+1} \right)^{(\sigma+2)/N} \\
		& = \Lambda(r_0+1)\left\| \nabla u^{(m+r)/2} \right\|_2^2 \left( \left( \|u\|_{r_0+1}^{r_0+1} \right)^{(\sigma+2)/N} - \frac{4mr_0}{\Lambda(r_0+1)(m+r)(m+r_0)} \right).
	\end{align*}
	Introducing
	\begin{equation*}
		C_0 := \left( \frac{4mr_0}{\Lambda(r_0+1)(m+r_0)^2} \right)^{N/(\sigma+2)},
	\end{equation*}
	the above inequality reads
	\begin{equation}
		\frac{1}{r+1} \frac{d}{dt} \|u\|_{r+1}^{r+1} \le \Lambda(r_0+1)\left\| \nabla u^{(m+r)/2} \right\|_2^2 \left( \left( \|u\|_{r_0+1}^{r_0+1} \right)^{(\sigma+2)/N} - \frac{m+r_0}{m+r} C_0^{(\sigma+2)/N} \right). \label{intgs04}
	\end{equation}
	
	\noindent\textbf{Step~1: $L^{r_0+1}$-estimate.}
	If
	\begin{equation}
		\|u_0\|_{r_0+1}^{r_0+1} < C_0, \label{intgs05}
	\end{equation}
	then we readily infer from~\eqref{intgs04} (with $r=r_0$) that
	\begin{equation}
		\|u(t)\|_{r_0+1}^{r_0+1} \le \|u_0\|_{r_0+1}^{r_0+1} < C_0, \qquad t\in [0,T_{\mathrm{max}}(u_0)). \label{intgs06}
	\end{equation}
	Another straightforward consequence of~\eqref{intgs05} is that there is $r_1(u_0)\in (r_0,r_0+1)$ such that
	\begin{equation}
		\|u_0\|_{r_0+1}^{r_0+1} \le C_0 \left( \frac{m+r_0}{m+r} \right)^{N/(\sigma+2)}, \qquad r\in [r_0,r_1(u_0)]. \label{intgs07}
	\end{equation}
	We then deduce from~\eqref{intgs04}, \eqref{intgs06}, and~\eqref{intgs07} that, for $r\in (r_0,r_1(u_0)]$,
	\begin{equation*}
		\frac{1}{r+1} \frac{d}{dt} \|u\|_{r+1}^{r+1} \le \Lambda(r_0+1)\left\| \nabla u^{(m+r)/2} \right\|_2^2 \left( \left( \|u_0\|_{r_0+1}^{r_0+1} \right)^{(\sigma+2)/N} - \frac{m+r_0}{m+r} C_0^{(\sigma+2)/N} \right) \le 0,
	\end{equation*}
	whence
\begin{equation}
		\|u(t)\|_{r+1}^{r+1} \le \|u_0\|_{r+1}^{r+1}, \qquad t\in [0,T_{\mathrm{max}}(u_0)), \ r\in [r_0,r_1(u_0)]. \label{intgs08}
\end{equation}
	
	\noindent\textbf{Step~2: $L^{r+1}$-estimate, $r\in (r_0,\infty)$.} The next step is to extend the above result to all $L^{r+1}$-norms with $r\ge r_0$. To this end, we adapt the argument in \cite[Theorem~1.1]{CPZ2004} and fix $r>r_1(u_0)$. Let $K>1$ to be determined later. Since $m>1$, it follows from~\eqref{eq1} and the algebraic inequality
\begin{equation*}
 (y+K)^p \le py^p+(pK)^p, \quad y>0, \quad p>1,
\end{equation*}
that
\begin{equation*}
\begin{split}
		\frac{1}{r+1} \frac{d}{dt} \|(u-K)_+\|_{r+1}^{r+1} & = - rm \int_{\mathbb{R}^N} (u-K)_+^{r-1} u^{m-1} \nabla (u-K)_+\cdot \nabla u\ dx \\
		& \qquad + \int_{\mathbb{R}^N} |x|^\sigma (u-K)_+^r u^p\ dx \\
		& \le - r m \int_{\mathbb{R}^N} (u-K)_+^{m+r-2} |\nabla(u-K)_+|^2\ dx \\
		& \qquad + \int_{\mathbb{R}^N} |x|^\sigma (u-K)_+^r \big[ (u-K)_+ + K \big]^p\ dx \\
		& \le - \frac{4rm}{(m+r)^2} \left\|\nabla(u-K)_+^{(m+r)/2} \right\|_2^2 \\
		& \qquad + \int_{\mathbb{R}^N} |x|^\sigma (u-K)_+^r \left[ p (u-K)_+^p + (pK)^p \right]\ dx,
\end{split}
\end{equation*}
and we further infer that
\begin{equation}
\begin{split}
			\frac{d}{dt} \|(u-K)_+\|_{r+1}^{r+1} &\leq - \frac{4m r(r+1)}{(m+r)^2} \left\| \nabla(u-K)_+^{(m+r)/2} \right\|_2^2 + p(r+1) \int_{\mathbb{R}^N} |x|^{\sigma}(u-K)_+^{r+p}\ dx \\
			& \qquad + (r+1) (pK)^p \int_{\mathbb{R}^N} |x|^\sigma (u-K)_+^r\ dx.
		\end{split}\label{gs08}
	\end{equation}
	We now estimate the last term on the right-hand side of~\eqref{gs08}. Since $\sigma\in (-N,0)$, it follows from H\"older's inequality that
	\begin{align*}
		(pK)^p \int_{\mathbb{R}^N} & |x|^\sigma (u-K)_+^r\ dx  \le (pK)^p \int_{B(0,1)} |x|^\sigma (u-K)_+^r\ dx + (pK)^p \int_{\mathbb{R}^N\setminus B(0,1)} (u-K)_+^r\ dx\\
		& \le (pK)^p \left( \int_{B(0,1)} |x|^\sigma (u-K)_+^{p+r}\ dx \right)^{r/(r+p)} \left( \int_{B(0,1)} |x|^\sigma\ dx \right)^{p/(r+p)} \\
		& \qquad + (pK)^p \left( \int_{\mathbb{R}^N} (u-K)_+^{r+1}\ dx \right)^{r/(r+1)} \Big|\{ u\ge K\}\Big|^{1/(r+1)}\\
		& \le C(r) K^p \left( \int_{\mathbb{R}^N} |x|^\sigma (u-K)_+^{p+r}\ dx \right)^{r/(r+p)} \\
		& \qquad + (pK)^p \|(u-K)_+\|_{r+1} ^r \left( \frac{\|u\|_{r_0+1}^{r_0+1}}{K^{r_0+1}} \right)^{1/(r+1)}.
	\end{align*}
Since $K>1$, we next infer from~\eqref{intgs06} and Young's inequality that
	\begin{align*}
		(pK)^p \int_{\mathbb{R}^N} |x|^\sigma (u-K)_+^r\ dx & \le \int_{\mathbb{R}^N} |x|^\sigma (u-K)_+^{p+r}\ dx + \|(u-K)_+\|_{r+1}^{r+1} + C(r) \left( K^{r+p} + K^{p(r+1)} \right).
	\end{align*}
	Combining~\eqref{gs08} and the above inequality leads us to
	\begin{align*}
		\frac{d}{dt} \|(u-K)_+\|_{r+1}^{r+1} & \le - \frac{4m r(r+1)}{(m+r)^2} \left\| \nabla(u-K)_+^{(m+r)/2} \right\|_2^2 + (p+1)(r+1) \int_{\mathbb{R}^N}|x|^{\sigma}(u-K)_+^{r+p}\ dx \\
		& \qquad + (r+1) \|(u-K)_+\|_{r+1}^{r+1} + C(r) K^{p(r+1)}.
	\end{align*}
	We now use Proposition~\ref{prop.gs2} to obtain
	\begin{align*}
		\frac{d}{dt} \|(u-K)_+\|_{r+1}^{r+1} & \le - \frac{4m r(r+1)}{(m+r)^2} \left\| \nabla(u-K)_+^{(m+r)/2} \right\|_2^2 \\
		& \qquad + (p+1)(r+1)\Lambda(r)\left\| \nabla(u-K)_+^{(m+r)/2} \right\|_2^2 \left( \|(u-K)_+\|_{r_0+1}^{r_0+1} \right)^{(\sigma+2)/N} \\
		& \qquad + (r+1) \|(u-K)_+\|_{r+1}^{r+1} + C(r) K^{p(r+1)},
	\end{align*}
	while~\eqref{intgs08} and H\"older's inequality ensure that
	\begin{align*}
		\|(u-K)_+\|_{r_0+1}^{r_0+1} & \le \frac{1}{K^{r_1(u_0)-r_0}} \int_{\mathbb{R}^N} u^{r_1(u_0)-r_0} (u-K)_+^{r_0+1}\ dx \\
		& \le \frac{\|u\|_{r_1(u_0)+1}^{r_1(u_0)+1}}{K^{r_1(u_0)-r_0}} \le \frac{\|u_0\|_{r_1(u_0)+1}^{r_1(u_0)+1}}{K^{r_1(u_0)-r_0}}.
	\end{align*}
	Therefore,
	\begin{align*}
		\frac{d}{dt} \|(u-K)_+\|_{r+1}^{r+1} & \le - \frac{4m r(r+1)}{(m+r)^2} \left\| \nabla(u-K)_+^{(m+r)/2} \right\|_2^2 \\
		& \qquad + C_1(r) K^{(r_0-r_1(u_0))(\sigma+2)/N} \left\| \nabla(u-K)_+^{(m+r)/2} \right\|_2^2 \\
		& \qquad + (r+1) \|(u-K)_+\|_{r+1}^{r+1} + C(r) K^{p(r+1)}.
	\end{align*}
	Choosing
	\begin{equation*}
		K=K_r := \left(\frac{C_1(r)(m+r)^2}{2 m r(r+1)}\right)^{N/[(\sigma+2)(r_1(u_0)-r_0)]},
	\end{equation*}
	we end up with
	\begin{equation}
	\begin{split}
		\frac{d}{dt} \|(u-K_r)_+\|_{r+1}^{r+1} & \le - \frac{2m r(r+1)}{(m+r)^2} \left\| \nabla(u-K_r)_+^{(m+r)/2} \right\|_2^2 \\
		& \qquad + (r+1) \|(u-K_r)_+\|_{r+1}^{r+1} + C(r).
	\end{split}\label{gs09}
	\end{equation}
	To complete the proof of the $L^r$-estimate, we use the following variant of the Gagliardo-Nirenberg inequality, see \cite[Lemma~2.4]{Su2006} for instance,
	\begin{equation*}
		\|w\|_{r+1} \le C(r) \left\| \nabla w^{(m+r)/2} \right\|_2^{2\Theta/(r+m)} \|w\|_{r_0+1}^{1-\Theta},
	\end{equation*}
	which is valid for any $w\in L_+^{r_0+1}(\mathbb{R}^N)$ with $\nabla w^{(m+r)/2}\in L^2(\mathbb{R}^N)$, the exponent $\Theta$ being given by
	\begin{equation*}
		\Theta := \frac{r+m}{r+1} \frac{N(r-r_0)}{N(r-r_0)+N(m-1)+2(r_0+1)}\in (0,1),
	\end{equation*}
the upper bound for $\Theta$ following from the fact that
\begin{equation*}
	1-\Theta=\frac{p-m}{(\sigma+2)(r+1)} \frac{N(m-1)+2(r+1)}{N(r-r_0)+N(m-1)+2(r_0+1)}>0.
\end{equation*}
	In particular, owing to~\eqref{intgs06},
	\begin{equation*}
		\|(u-K_r)_+\|_{r+1}^{(r+m)/\Theta} \le C(r) \left\| \nabla (u-K_r)_+^{(m+r)/2} \right\|_2^{2},
	\end{equation*}
	and we combine the above inequality with~\eqref{gs09} to obtain
	\begin{equation*}
		\frac{d}{dt} \|(u-K_r)_+\|_{r+1}^{r+1} + 2C_2(r) \|(u-K_r)_+\|_{r+1}^{(r+m)/\Theta} \le C_3(r) \left( 1 + \|(u-K_r)_+\|_{r+1}^{r+1} \right).
	\end{equation*}
	Since
	\begin{equation*}
		\theta := \frac{r+1}{r+m} \Theta <1,
	\end{equation*}
	we deduce from Young's inequality that
	\begin{equation*}
		\frac{d}{dt} \|(u-K_r)_+\|_{r+1}^{r+1} + C_2(r) \|(u-K_r)_+\|_{r+1}^{(r+1)/\theta} \le C(r),
	\end{equation*}
	from which we readily infer that
	\begin{equation}
		\|(u(t)-K_r)_+\|_{r+1}^{r+1} \le C(r), \qquad t\in [0,T_{\mathrm{max}}(u_0)). \label{gs10}
	\end{equation}
	Finally,
	\begin{equation*}
		u(t,x)^{r+1} \le (2K_r)^{r-r_0} u(t,x)^{r_0+1} + 2^{r+1} (u(t,x)-K_r)_+^{r+1}, \qquad (t,x)\in [0,T_{\mathrm{max}}(u_0))\times\mathbb{R}^N,
	\end{equation*}
	and it follows from~\eqref{intgs06}, \eqref{gs10}, and the above inequality that
	\begin{equation}
		\|u(t)\|_{r+1}^{r+1} \le C(r), \qquad t\in [0,T_{\mathrm{max}}(u_0)). \label{gs11}
	\end{equation}
	
	\noindent\textbf{Step~3: $L^{\infty}$-estimate.} Taking into account the previous uniform $L^r$-estimates for $u(t)$, $t\in [0,T_{\mathrm{max}}(u_0))$, we can repeat the proof of Proposition~\ref{prop.ex3} to obtain that, for any $T>0$, there is $C(T)>0$ such that
	\begin{equation*}
		\|u(t)\|_\infty \le C(T), \qquad t\in [0,T_{\mathrm{max}}(u_0)) \cap [0,T].
	\end{equation*}
	Consequently, the $L^\infty$-norm of $u$ cannot blow up in finite time and we infer from Corollary~\ref{cor.maxwp} that $T_{\mathrm{max}}(u_0)=\infty$, thereby completing the proof of Theorem~\ref{th.global}.
\end{proof}

\section{Unboundedness: $p\in (1,p_G]$}\label{sec.ub}

We conclude the paper with this final section, dedicated to the proof of Theorem~\ref{th.unbdd}.

\begin{proof}[Proof of Theorem~\ref{th.unbdd}] We split the proof into four steps, first showing the stated unboundedness by a contradiction argument. We next derive lower and upper bounds on the grow-up rates by comparison with the self-similar solutions constructed in \cite{ILS24b,IMS23}, the upper bound requiring additionally the initial condition to be compactly supported. We finish off the proof with the convergence to self-similarity when $p\in (1,p_G)$, which is actually an immediate consequence of the already established upper and lower bounds.

\medskip

\noindent\textbf{Step 1. Unboundedness.} Let $p\in (1,p_G] \subset(1,m)$ and $u_0\in L_+^1(\mathbb{R}^N)\cap L^\infty(\mathbb{R} ^N)$, $u_0\not\equiv 0$. According to Corollary~\ref{cor.exist}, there is at least one global solution $u$ to~\eqref{cp}. We set
\begin{equation}
	M_\infty := \sup_{t\ge 0} \|u(t)\|_\infty \in [\|u_0\|_\infty, \infty], \label{x03}
\end{equation}
and assume for contradiction that $M_\infty<\infty$. Introducing $\lambda = M_\infty^{(m-p)/(2+\sigma)}$ and
\begin{equation*}
	V(t,x) = u\big(\lambda^2 t,\lambda x\big), \qquad (t,x)\in [0,\infty)\times \mathbb{R}^N,
\end{equation*}
we compute
\begin{align*}
	\partial_t V(t,x) - \Delta V^m(t,x) - |x|^\sigma V^m(t,x) & = |x|^\sigma V^p(t,x) \left( \lambda^{2+\sigma} - u^{m-p} \big(\lambda^2 t,\lambda x\big) \right) \\
	& \ge |x|^\sigma V^p(t,x) \left( \lambda^{2+\sigma} - \|u\big(\lambda^2 t\big)\|_\infty^{m-p} \right)
\end{align*}
for $(t,x)\in (0,\infty)\times \mathbb{R}^N$, where we have used the positivity of $m-p$ to derive the lower bound. Owing to~\eqref{x03} and the choice of $\lambda$, we deduce from the previous computation that $V$ is a supersolution to~\eqref{cp} with $m$ instead of $p$ and initial condition $x\mapsto u_0(\lambda x)$. As $m>p_G$, we may use the comparison principle established in Theorem~\ref{th.uniq} to deduce that
\begin{equation}
	V(t,x) \ge v(t,x), \qquad (t,x)\in [0,\infty)\times\mathbb{R}^N, \label{x04}
\end{equation}
where $v$ solves
\begin{equation}
\begin{aligned}
		\partial_t v(t,x) & = \Delta v^m(t,x) + |x|^\sigma v^m(t,x), \qquad (t,x)\in [0,\infty)\times\mathbb{R}^N, \\
		v(0,x) & = u_0(\lambda x), \qquad \ x\in\real^N.
\end{aligned}\label{x05}
\end{equation}
However, $v$ blows up in finite time according to Theorem~\ref{th.blowup} since $m\in (p_G,p_F)$, which contradicts~\eqref{x04}. Consequently, $M_\infty=\infty$ as claimed.

\smallskip

\noindent \textbf{Step 2. Lower bound on the grow-up rate.} Let us now consider $p\in(1,p_G]$ and $u_0\in L_+^1(\mathbb{R}^N)\cap L^\infty(\mathbb{R}^N)$, $u_0\not\equiv 0$. It readily follows from the comparison principle that
\begin{equation*}
	u(t,x) \ge U(t,x), \qquad (t,x)\in [0,\infty)\times\mathbb{R}^N,
\end{equation*}
where $U$ denotes the unique solution to the Cauchy problem for the porous medium equation~\eqref{pme}. By \cite[Proposition~9.19]{Va2007}, there is $T_1>0$ such that $B(0,2)\subset \mathrm{supp}\, U(T_1) \subset \mathrm{supp}\, u(T_1)$ and the continuity of $U(T_1)$ implies that there is $\delta_1\in (0,1)$ such that
\begin{equation}
	0< \delta_1 \le U(T_1,x) \le u(T_1,x), \qquad x\in B(0,1). \label{x05.5}
\end{equation}

We first handle the case $p\in (1,p_G)$ and pick $t_1\in (0,T_1)$ sufficiently small such that
\begin{equation*}
	t_1^{\alpha_*} \|f_*\|_{L^\infty([0,\infty))} \le \delta_1 \;\;\text{ and }\;\; \varrho_* t_1^{\beta_*} \le 1,
\end{equation*}
where $f_*$ is the profile of the unique forward self-similar solution to~\eqref{eq1} constructed in \cite{IMS23}, see~\eqref{forward}, and $[0,\varrho_*] = \mathrm{supp}\, f_*$. Then, for $x\in B\big(0,\varrho_* t_1^{\beta_*}\big)$, we deduce from~\eqref{x05.5} that
\begin{equation*}
	\mathcal{U}_*(t_1,x) = t_1^{\alpha_*} f_*\big(|x| t_1^{-\beta_*}\big) \le t_1^{\alpha_*} \|f_*\|_{L^\infty([0,\infty))} \le \delta_1 \le u(T_1,x),
\end{equation*}
while, for $x\in\mathbb{R}^N\setminus B\big(0,\varrho_* t_1^{\beta_*}\big)$,
\begin{equation*}
	\mathcal{U}_*(t_1,x) = t_1^{\alpha_*} f_*\big(|x| t_1^{-\beta_*}\big) = 0 \le u(T_1,x).
\end{equation*}
Since $\mathcal{U}_*(t_1)$and $u(T_1)$ are both positive in a neighborhood of $x=0$, we are in a position to apply the comparison principle established in Theorem~\ref{th.uniq} to conclude that
\begin{equation}
	\mathcal{U}_*(t+t_1,x) = (t_1+t)^{\alpha_*} f_*\big(|x| (t_1+t)^{-\beta_*}\big) \le u(T_1+t,x), \qquad (t,x)\in [0,\infty)\times\mathbb{R}^N. \label{x06a}
\end{equation}
In particular,
\begin{equation*}
	\|u(T_1+t)\|_\infty \ge (t_1+t)^{\alpha_*} \|f_*\|_{L^\infty([0,\infty))}, \qquad t\ge 0,
\end{equation*}
from which the first lower bound in~\eqref{lowbd} follows.

Similarly, when $p=p_G$, let
\begin{equation*}
	\mathcal{U}^*(t,x)=e^{\alpha^* t}f^*\big(|x|e^{-\beta^* t}\big), \qquad (t,x)\in(0,\infty)\times\mathbb{R}^N,
\end{equation*}
be one of the self-similar solutions in exponential form~\eqref{exponential} classified in \cite{ILS24b}. Similar arguments as above lead to the existence of $t_1\in (-\infty,T_1)$ such that the estimate
\begin{equation*}
	e^{\alpha^*(t+t_1)}f^*\big(|x|e^{-\beta^*(t+t_1)}\big)\leq u(t+T_1,x), \quad (t,x)\in[0,\infty)\times\mathbb{R}^N,
\end{equation*}
holds true, and thus to the second lower bound in~\eqref{lowbd}.

\smallskip

\noindent\textbf{Step 3. Upper bound on the grow-up rate.} We now assume that $u_0$ is compactly supported; that is, there is $R_0>0$ such that $\mathrm{supp}\, u_0\subset B(0,R_0)$. As in the previous step, we begin with the case $p\in (1,p_G)$. Since the support of $\mathcal{U}_*(t)$ increases with time, there is $T_0>0$ large enough such that
\begin{equation*}
	\varrho_* T_0^{\beta_*} \ge 2 R_0 \;\;\text{ and }\;\; T_0^{\alpha_*} \inf_{[0,\varrho_*/2)} f_* \ge \|u_0\|_\infty.
\end{equation*}
Then, either $|x| \le \varrho_* T_0^{\beta_*}/2$ and
\begin{equation*}
	T_0^{\alpha_*} f_*\big(|x|T_0^{-\beta_*}\big) \ge \|u_0\|_\infty \ge u_0(x).
\end{equation*}
Or $|x|>\varrho_* T_0^{\beta_*}/2$ and $u_0(x)=0 \le T_0^{\alpha_*} f_*\big(|x|T_0^{-\beta_*}\big)$. Owing to the positivity of $\mathcal{U}_*(T_0)$ and $u_0$ in a neighborhood of $x=0$, we may apply Theorem~\ref{th.uniq} to conclude that
\begin{equation}
	u(t,x) \le \mathcal{U}_*(T_0+t,x) = (T_0+t)^{\alpha_*} f_*\big(|x| (T_0+t)^{-\beta_*}\big), \qquad (t,x)\in [0,\infty)\times\mathbb{R}^N. \label{x06b}
\end{equation}
Combining~\eqref{x06b} and the boundedness of $f_*$ gives the first upper bound in~\eqref{upbd}. Similarly, when $p=p_G$, we may proceed as above and find $T_0>0$ large enough such that
\begin{equation*}
	u(t,x)\leq e^{\alpha^*(t+T_0)}f^*\big(|x|e^{-\beta^*(t+T_0)}\big), \quad (t,x)\in[0,\infty)\times\mathbb{R}^N,
\end{equation*}
which leads to the second upper bound in~\eqref{upbd}.

\smallskip

\noindent\textbf{Step~4. Convergence to self-similarity.} Let $p\in(1,p_G)$ and $u_0\in L_+^1(\mathbb{R}^N)\cap L^\infty(\mathbb{R}^N)$, $u_0\not\equiv 0$, with compact support. It follows from the above analysis that $u$ satisfies the pointwise estimates~\eqref{x06a} and~\eqref{x06b}. We then infer from~\eqref{x06a} that, for $t\ge T_1$ and $x\in B(0,\varrho_*(t+t_1-T_1)^{\beta_*})$,
\begin{align*}
	t^{-\alpha_*} & \big[ \mathcal{U}_*(t,x) - u(t,x) \big] \le t^{-\alpha_*} \left[ \mathcal{U}_*(t,x) - (t+t_1-T_1)^{\alpha_*} f_*\big( |x| (t+t_1-T_1)^{-\beta_*} \big) \right]\\
	& = f_*\big( |x|t^{-\beta_*} \big) - \left( 1 - \frac{T_1-t_1}{t} \right)^{\alpha_*} f_*\big( |x| (t+t_1-T_1)^{-\beta_*} \big) \\
	& = \left[ 1 - \left( 1 - \frac{T_1-t_1}{t} \right)^{\alpha_*} \right] f_*\big( |x| (t+t_1-T_1)^{-\beta_*} \big) + f_*\big( |x| t^{-\beta_*} \big) - f_*\big( |x| (t+t_1-T_1)^{-\beta_*} \big) \\
	& \le \max\{1,\alpha_*\} \left( \frac{T_1-t_1}{t} \right)^{\min\{1,\alpha_*\}} \|f_*\|_{L^\infty([0,\infty))} + |x| \left( (t+t_1-T_1)^{-\beta_*} - t^{-\beta_*} \right) \|f_*'\|_{L^\infty([0,\infty))} \\
	& \le C t^{- \min\{1,\alpha_*\}} + \varrho_* \left[1- \left( 1 - \frac{T_1-t_1}{t}\right)^{\beta_*}\right] \|f_*'\|_{L^\infty([0,\infty))}\\
	& \le C \left( t^{- \min\{1,\alpha_*\}} + t^{- \min\{1,\beta_*\}}  \right).
\end{align*}
Since $\mathcal{U}_*(t,x) - u(t,x)\le 0$ for $t\ge T_1$ and $x\not\in B(0,\varrho_*(t+t_1-T_1)^\beta)$, we conclude that
\begin{equation}
	\sup_{t\ge T_1} t^{-\alpha_*} \big[ \mathcal{U}_*(t,x) - u(t,x) \big] \le C \left( t^{- \min\{1,\alpha_*\}} + t^{- \min\{1,\beta_*\}}  \right). \label{sup04}
\end{equation}
A similar computation allows us to deduce from~\eqref{x06b} that
\begin{align*}
	t^{-\alpha} & \big[ u(t,x) - \mathcal{U}_*(t,x) \big] \le \left[ \left( 1 + \frac{T_0}{t} \right)^\alpha -  1 \right] \|f_*\|_{L^\infty([0,\infty))} + \varrho_* \left[ \left( 1 + \frac{T_0}{t}\right)^\beta - 1 \right] \|f_*'\|_{L^\infty([0,\infty))}
\end{align*}
for $(t,x)\in (0,\infty)\times\mathbb{R}^N$, whence
\begin{equation}
	\sup_{t\ge T_0} t^{-\alpha_*} \big[ u(t,x) - \mathcal{U}_*(t,x) \big] \le C \left( t^{- \min\{1,\alpha_*\}} + t^{- \min\{1,\beta_*\}}  \right). \label{sup05}
\end{equation}
Combining~\eqref{sup04} and~\eqref{sup05} gives the claim~\eqref{x02}.
\end{proof}

\begin{remark} \label{rem.cvss}
Observe that, when $p=p_G$, \textbf{Step~4} in the above proof cannot be performed and it is unclear whether convergence to self-similarity holds true in that case as well.
\end{remark}

\section*{Acknowledgements} This work is partially supported by the Spanish project PID2020-115273GB-I00 and by the Grant RED2022-134301-T (Spain). Part of this work has been developed during visits of R. G. I. to Laboratoire de Math\'ematiques LAMA, Universit\'e de Savoie, and R. G. I. wants to thank this institution for hospitality and support. The authors wish to thank Professor Dong Ye for interesting discussions and pointing us out some new references.

\bigskip

\noindent \textbf{Data availability} Our manuscript has no associated data.

\bigskip

\noindent \textbf{Conflict of interest} The authors declare that there is no conflict of interest.

\bibliographystyle{siam}
\bibliography{WP_PME_Henon}

\begin{thebibliography}{10}

\bibitem{AL79}
{\sc H.~W. Alt and S.~Luckhaus}, {\em Quasilinear elliptic-parabolic
  differential equations}, Math. Z., 183 (1983), pp.~311--341.

\bibitem{AdB91}
{\sc D.~Andreucci and E.~DiBenedetto}, {\em On the {C}auchy problem and initial
  traces for a class of evolution equations with strongly nonlinear sources},
  Ann. Scuola Norm. Sup. Pisa Cl. Sci. (4), 18 (1991), pp.~363--441.

\bibitem{AT05}
{\sc D.~Andreucci and A.~F. Tedeev}, {\em Universal bounds at the blow-up time
  for nonlinear parabolic equations}, Adv. Differential Equations, 10 (2005),
  pp.~89--120.

\bibitem{BG84}
{\sc P.~Baras and J.~A. Goldstein}, {\em The heat equation with a singular
  potential}, Trans. Amer. Math. Soc., 284 (1984), pp.~121--139.

\bibitem{BSTW17}
{\sc B.~Ben~Slimene, S.~Tayachi, and F.~B. Weissler}, {\em Well-posedness,
  global existence and large time behavior for {H}ardy-{H}\'enon parabolic
  equations}, Nonlinear Anal., 152 (2017), pp.~116--148.

\bibitem{BeHo1980}
{\sc J.~G. Berryman and C.~J. Holland}, {\em Stability of the separable
  solution for fast diffusion}, Arch. Rational Mech. Anal., 74 (1980),
  pp.~379--388.

\bibitem{CKN1984}
{\sc L.~Caffarelli, R.~Kohn, and L.~Nirenberg}, {\em First order interpolation
  inequalities with weights}, Compositio Math., 53 (1984), pp.~259--275.

\bibitem{Ca99}
{\sc J.~Carrillo}, {\em Entropy solutions for nonlinear degenerate problems},
  Arch. Ration. Mech. Anal., 147 (1999), pp.~269--361.

\bibitem{CIT21}
{\sc N.~Chikami, M.~Ikeda, and K.~Taniguchi}, {\em Well-posedness and global
  dynamics for the critical {H}ardy-{S}obolev parabolic equation},
  Nonlinearity, 34 (2021), pp.~8094--8142.

\bibitem{CIT22}
\leavevmode\vrule height 2pt depth -1.6pt width 23pt, {\em Optimal
  well-posedness and forward self-similar solution for the {H}ardy-{H}\'enon
  parabolic equation in critical weighted {L}ebesgue spaces}, Nonlinear Anal.,
  222 (2022).
\newblock Article no.~112931.

\bibitem{CITT24}
{\sc N.~Chikami, M.~Ikeda, K.~Taniguchi, and S.~Tayachi}, {\em Unconditional
  uniqueness and non-uniqueness for {H}ardy--{H}\'enon parabolic equations},
  Math. Ann., 390 (2024), pp.~3765--3825.

\bibitem{CPZ2004}
{\sc L.~Corrias, B.~Perthame, and H.~Zaag}, {\em Global solutions of some
  chemotaxis and angiogenesis systems in high space dimensions}, Milan J.
  Math., 72 (2004), pp.~1--28.

\bibitem{GP76}
{\sc B.~H. Gilding and L.~A. Peletier}, {\em The {C}auchy problem for an
  equation in the theory of infiltration}, Arch. Rational Mech. Anal., 61
  (1976), pp.~127--140.

\bibitem{He73}
{\sc M.~H\'enon}, {\em Numerical experiments on the stability of spherical
  stellar systems}, Astron. \& Astrophys., 24 (1973), pp.~229--238.

\bibitem{HS24}
{\sc K.~Hisa and M.~Sierzega}, {\em Existence and nonexistence of solutions to
  the {H}ardy parabolic equation}, Funkcial. Ekvac., 67 (2024), pp.~149--174.

\bibitem{HT21}
{\sc K.~Hisa and J.~Takahashi}, {\em Optimal singularities of initial data for
  solvability of the {H}ardy parabolic equation}, J. Differential Equations,
  296 (2021), pp.~822--848.

\bibitem{ILS24a}
{\sc R.~G. Iagar, M.~Latorre, and A.~S\'anchez}, {\em Blow-up patterns for a
  reaction-diffusion equation with weighted reaction in general dimension},
  Adv. Differential Equations, 29 (2024), pp.~515--574.

\bibitem{ILS24b}
\leavevmode\vrule height 2pt depth -1.6pt width 23pt, {\em Eternal solutions in
  exponential self-similar form for a quasilinear reaction-diffusion equation
  with critical singular potential}, Discrete Contin. Dyn. Syst., 44 (2024),
  pp.~1329--1353.

\bibitem{IL}
{\sc R.~G. Iagar and {\relax Ph}.~Lauren\c{c}ot}, {\em A {H}ardy-{H}énon
  equation in $\mathbb{R}^n$ with sublinear absorption}, Calc. Var. Partial
  Differ. Equ., 64 (2025).
\newblock Article no.~74.

\bibitem{IMS23}
{\sc R.~G. Iagar, A.~I. Mu\~noz, and A.~S\'anchez}, {\em Self-similar solutions
  preventing finite time blow-up for reaction-diffusion equations with singular
  potential}, J. Differential Equations, 358 (2023), pp.~188--217.

\bibitem{Ka1963}
{\sc S.~Kaplan}, {\em On the growth of solutions of quasi-linear parabolic
  equations}, Comm. Pure Appl. Math., 16 (1963), pp.~305--330.

\bibitem{Ka93}
{\sc T.~Kawanago}, {\em Estimation {$L^p$}--{$L^q$} des solutions de
  {$u_t=\Delta\phi(u)+f$}}, C. R. Acad. Sci. Paris S\'er. I Math., 317 (1993),
  pp.~821--824.

\bibitem{Le1973}
{\sc H.~A. Levine}, {\em Some nonexistence and instability theorems for
  solutions of formally parabolic equations of the form
  {$Pu_{t}=-Au+\mathscr{F}(u)$}}, Arch. Rational Mech. Anal., 51 (1973),
  pp.~371--386.

\bibitem{LiYan2023}
{\sc Y.~Li and X.~Yan}, {\em Anisotropic {Caffarelli}-{Kohn}-{Nirenberg} type
  inequalities}, Adv. Math., 419 (2023).
\newblock Article no.~108958.

\bibitem{MiPo2001}
{\sc E.~Mitidieri and S.~I. Pohozaev}, {\em A priori estimates and blow-up of
  solutions to nonlinear partial differential equations and inequalities},
  Proc. Steklov Inst. Math., 234 (2001), pp.~1--362.

\bibitem{Ot96}
{\sc F.~Otto}, {\em {$L^1$}-contraction and uniqueness for quasilinear
  elliptic-parabolic equations}, J. Differential Equations, 131 (1996),
  pp.~20--38.

\bibitem{Qi98}
{\sc Y.-w. Qi}, {\em The critical exponents of parabolic equations and blow-up
  in {${\bf R}^n$}}, Proc. Roy. Soc. Edinburgh Sect. A, 128 (1998),
  pp.~123--136.

\bibitem{QuSo2019}
{\sc P.~Quittner and P.~Souplet}, {\em Superlinear parabolic problems. Blow-up,
  global existence and steady states}, Birkh\"auser Advanced Texts: Basler
  Lehrb\"ucher. [Birkh\"auser Advanced Texts: Basel Textbooks],
  Birkh\"auser/Springer, Cham, second~ed., 2019.

\bibitem{SGKM1995}
{\sc A.~A. Samarskii, V.~A. Galaktionov, S.~P. Kurdyumov, and A.~P. Mikhailov},
  {\em Blow-up in quasilinear parabolic equations}, vol.~19 of De Gruyter
  Expositions in Mathematics, Walter de Gruyter \& Co., Berlin, 1995.
\newblock Translated from the 1987 Russian original by Michael Grinfeld and
  revised by the authors.

\bibitem{Su2006}
{\sc Y.~Sugiyama}, {\em Global existence in sub-critical cases and finite time
  blow-up in super-critical cases to degenerate {K}eller-{S}egel systems},
  Differential Integral Equations, 19 (2006), pp.~841--876.

\bibitem{Su02}
{\sc R.~Suzuki}, {\em Existence and nonexistence of global solutions of
  quasilinear parabolic equations}, J. Math. Soc. Japan, 54 (2002),
  pp.~747--792.

\bibitem{Va2007}
{\sc J.~L. V\'azquez}, {\em The porous medium equation}, Oxford Mathematical
  Monographs, The Clarendon Press, Oxford University Press, Oxford, 2007.
\newblock Mathematical theory.

\end{thebibliography}

\end{document}